\documentclass[11pt]{article}
\usepackage{amsmath,amsfonts,amssymb,amsthm,enumerate,graphicx}
\usepackage{tikz,authblk}
\usepackage{subfig}
\usepackage{url}
\usepackage{epic}

\newtheorem{theorem}{{\bf Theorem}}[section]

\newtheorem{definition}[theorem]{{\bf Definition}}
\newtheorem{example}[theorem]{{\bf Example}}
\newtheorem{lemma}[theorem]{{\bf Lemma}}

\newtheorem{proposition}[theorem]{{\bf Proposition}}
\newtheorem{remark}[theorem]{{\bf Remark}}

\newcommand{\lk}[2]{{\rm lk}_{#1}(#2)}
\newcommand{\st}[2]{{\rm st}_{#1}(#2)}
\newcommand{\skel}[2]{{\rm skel}_{#1}(#2)}
\newcommand{\Kd}[1]{{\mathcal K}(#1)}
\newcommand{\bs}{\backslash}
\newcommand{\ol}{\overline}
\newcommand{\meets}{\leftrightarrow}
\newcommand{\nmeets}{\nleftrightarrow}

\newcommand{\FF}{ \ensuremath{\mathbb{F}}}
\newcommand{\QQ}{ \ensuremath{\mathbb{Q}}}
\newcommand{\ZZ}{ \ensuremath{\mathbb{Z}}}

\newcommand{\IntA}{A^{\!\!\!^{^{\circ}}}}
\newcommand{\IntB}{B^{\!\!\!^{^{\circ}}}}

\newcommand{\TPSS}{S^{\hspace{.2mm}2}\! \times \hspace{-3.3mm}_{-} \,
S^{\hspace{.1mm}1}}

\renewcommand{\Authfont}{\scshape\small}
\renewcommand{\Affilfont}{\itshape\small}
\renewcommand{\Authand}{,}
\renewcommand{\Authands}{, }

\begin{document}

\author[1] {Basudeb Datta}
\author[2] { Dipendu Maity}

\affil[1,2]{Department of Mathematics, Indian Institute of Science, Bangalore 560\,012, India.
 ${}^1$dattab@math.iisc.ernet.in, ${}^2$dipendumaity16@math.iisc.ernet.in.}

\title{Semi-equivelar maps on the torus are Archimedean}


\date{June 29, 2017}

\maketitle

\vspace{-10mm}

\begin{abstract}

If the face-cycles at all the vertices in a map on a surface are of same type then the map is called semi-equivelar. There are eleven types of Archimedean tilings on the plane. All the Archimedean tilings are semi-equivelar maps.
If a map $X$ on the torus is a quotient of an Archimedean tiling on the plane then the map $X$ is semi-equivelar. We show that each semi-equivelar map on the torus is a quotient of an Archimedean tiling on the plane.

Vertex-transitive maps are semi-equivelar maps. We know that four types of semi-equivelar maps on the torus are always vertex-transitive and there are examples of other seven types of semi-equivelar maps which are not vertex-transitive. We show that the number of  ${\rm Aut}(Y)$-orbits of vertices for any semi-equivelar map $Y$ on the torus is at most six. In fact, the number of orbits is at most three except one type of semi-equivelar maps. Our bounds on the number of orbits are sharp.
\end{abstract}

\noindent {\small {\em MSC 2010\,:} 52C20, 52B70, 51M20, 57M60.

}

\section{Introduction}

Here all the maps are polyhedral maps on surfaces. Thus, a face of a map is an $n$-gon for some integer $n\geq 3$ and intersection of two faces of a map is either empty or a vertex or an edge. A map $M$ is said to be {\em vertex-transitive} if the automorphism group ${\rm Aut}(M)$ acts transitively on the set $V(M)$ of vertices of $M$.

For a vertex $u$ in a map $M$, the faces containing $u$ form a cycle (called the {\em face-cycle} at $u$) $C_u$ in the dual graph $\Lambda(M)$ of $M$. Clearly, $C_u$ is of the form $P_1\mbox{-}P_2\mbox{-}\cdots\mbox{-}P_k\mbox{-}F_{1,1}$, where $P_i=F_{i,1}\mbox{-}\cdots \mbox{-}F_{i,n_i}$ is a path consisting of $p_i$-gons $F_{i,1}, \dots, F_{i,n_i}$, $p_i\neq p_{i+1}$ for $1\leq i\leq k-1$ and $p_n\neq p_1$. A map $M$ is called {\em semi-equivelar} if $C_u$ and $C_v$ are of same type for any two vertices $u$ and $v$ of $M$. More precisely, there exist integers $p_1, \dots, p_k\geq 3$ and $n_1, \dots, n_k\geq 1$, $p_i\neq p_{i+1}$ (addition in the suffix is modulo $k$) such that both $C_u$ and $C_v$ are of the form $P_1\mbox{-}P_2\mbox{-}\cdots\mbox{-}P_k\mbox{-}F_{1,1}$ as above, where $P_i=F_{i,1}\mbox{-}\cdots \mbox{-}F_{i,n_i}$ is a path consisting of $n_i$ $p_i$-gons. In this case, we say that $M$ is {\em semi-equivelar of type} $[p_1^{n_1}, \dots, p_k^{n_k}]$. (We identify $[p_1^{n_1}, \dots, p_k^{n_k}]$ with $[p_k^{n_k}, \dots, p_1^{n_1}]$ and with $[p_2^{n_2}, \dots, p_k^{n_k}, p_1^{n_1}]$.)
A semi-equivelar map of type $[p^q]$ is also called an {\em equivelar} map.
Clearly, vertex-transitive maps are semi-equivelar.


An {\em Archimedean} tiling of the plane $\mathbb{R}^2$ is a tiling of $\mathbb{R}^2$ by regular polygons such that all the vertices of the tiling are of same type.
An {\em Archimedean} tiling of $\mathbb{R}^2$ is also known as a {\em semi-regular}, or {\em homogeneous}, or {\em uniform} tiling.
In \cite{GS1977}, Gr\"{u}nbaum and Shephard showed that there are exactly eleven types of Archimedean tilings on the plane (see Example \ref{exam:plane}). These types are $[3^6]$, $[4^4]$, $[6^3]$, $[3^4,6^1]$, $[3^3,4^2]$,  $[3^2,4^1,3^1,4^1]$, $[3^1,6^1,3^1,6^1]$, $[3^1,4^1,6^1,4^1]$, $[3^1,12^2]$,  $[4^1,6^1,12^1]$, $[4^1,8^2]$.
Clearly, an Archimedean tiling on $\mathbb{R}^2$ is a semi-equivelar map on $\mathbb{R}^2$. But, there are semi-equivelar maps on $\mathbb{R}^2$ which are not (not isomorphic to) Archimedean tilings. In fact, there exists $[p^q]$ equivelar maps on $\mathbb{R}^2$ whenever $1/p+1/q\leq 1/2$ (e.g., \cite{CM1957}, \cite{FT1965}). It was shown in \cite{DU2005} and \cite{DM2017} that the Archimedean tilings $E_1, E_4, E_5, E_6$ are unique as semi-equivelar maps. Here we prove this for remaining Archimedean tilings. More precisely, we have

\begin{theorem} \label{theo:plane}
Let $E_1, \dots, E_{11}$ be the Archimedean tilings on the plane given in Example $\ref{exam:plane}$. Let $X$ be a semi-equivelar map on the plane. If the type of $X$ is same as the type of $E_i$, for some $i\leq 11$, then $X\cong E_i$. In particular, $X$ is vertex-transitive.
\end{theorem}

There are infinitely many semi-equivelar maps on both the torus and the Klein bottle (e.g., \cite{DN2001}, \cite{DU2005}). But, there are only eleven types of semi-equivelar maps on the torus and ten types of semi-equivelar maps on the Klein bottle (\cite{DM2017}). All the known examples are quotients of Archimedean tilings of the plane \cite{Ba1991, Su2011t, Su2011kb}. This motivates us to define

\begin{definition} \label{def:archimedean}
{\rm A semi-equivelar map on the torus or on the Klein bottle is said to be an} Archimedean map {\rm if it is the quotient of an Archimedean tiling on the plane by a discrete subgroup of the automorphism group of the tiling.
}
\end{definition}

As a consequence of Theorem \ref{theo:plane} we prove

\begin{theorem} \label{theo:sem=archimedian}
All 
semi-equivelar maps on the torus and the Klein bottle are Archimedean.
\end{theorem}

We know from \cite{DM2017} and \cite{DU2005}

\begin{proposition} \label{prop:vtmaps}
Let $X$ be a semi-equivelar map on the torus. If the type of $X$ is $[3^6]$, $[6^3]$,     $[4^4]$ or $[3^3,4^2]$ then $X$ is vertex-transitive.
\end{proposition}

\begin{proposition} \label{prop:non-vtmaps}
If $[p_1^{n_1}, \dots, p_k^{n_k}] = [3^2,4^1,3^1,4^1]$, $[3^4,6^1]$, $[3^1,6^1,3^1,6^1]$, $[3^1,4^1,6^1,4^1]$, \linebreak  $[3^1,12^2]$, $[4^1,8^2]$  or $[4^1,6^1,12^1]$ then there exists a semi-equivelar map of type $[p_1^{n_1}, \dots, p_k^{n_k}]$ on the torus which is not vertex-transitive.
\end{proposition}

\begin{proposition} \label{prop:no-of-orbits}
Let $X$ be a semi-equivelar map on the torus. Let the vertices of $X$ form $m$ ${\rm Aut}(X)$-orbits. {\rm (a)} If the type of $X$ is  $[3^2,4^1,3^1,4^1]$ then $m\leq 2$.
{\rm (b)} If the type of $X$ is  $[3^1,6^1,3^1,6^1]$ then $m\leq 3$.
\end{proposition}

Here we present bounds on the number of orbits for the remaining five cases. We prove

\begin{theorem} \label{theo:no-of-orbits}
Let $X$ be a semi-equivelar map on the torus. Let the vertices of $X$ form $m$ ${\rm Aut}(X)$-orbits. {\rm (a)} If the type of $X$ is  $[4^1,8^2]$ then $m\leq 2$.
{\rm (b)} If the type of $X$ is  $[3^4,6^1]$, $[3^1,4^1,6^1,4^1]$ or $[3^1,12^2]$ then $m\leq 3$. {\rm (c)} If the type of $X$ is  $[4^1,6^1, 12^1]$ then $m\leq 6$.
\end{theorem}

It follows from Proposition \ref{prop:no-of-orbits} and Theorem \ref{theo:no-of-orbits} that a semi-equivelar map $X$ of type $[3^2,4^1,3^1,4^1]$ or $[4^1,8^2]$ is either vertex-transitive or the number of ${\rm Aut}(X)$-orbits of vertices is two. In \cite[Example 4.2]{DM2017}, we have presented such semi-equivelar maps which are not vertex-transitive. Now, we present

\begin{theorem} \label{theo:orbits-3&6}
{\rm (a)} If $[p_1^{n_1}, \dots, p_k^{n_k}] = [3^1,6^1,3^1,6^1]$, $[3^4,6^1]$, $[3^1,4^1,6^1, 4^1]$ or $[3^1, 12^2]$ then there exists a semi-equivelar map $M$ of type $[p_1^{n_1}, \dots, p_k^{n_k}]$ on the torus with exactly three ${\rm Aut}(M)$-orbits of vertices.
{\rm (b)} There exists a semi-equivelar map $N$ of type $[4^1,6^1, 12^1]$ on the torus with exactly six ${\rm Aut}(N)$-orbits of vertices.
\end{theorem}

Thus, all the bounds in Proposition \ref{prop:no-of-orbits} and Theorem \ref{theo:no-of-orbits} are sharp.

\section{Examples} \label{sec:examples}

We first present eleven Archimedean tilings on the plane and five semi-equivelar maps on the torus. We need these examples for the proofs of our results in Sections \ref{sec:proofs-1} and  \ref{sec:proofs-2}.

\begin{example} \label{exam:plane}
{\rm Eleven Archimedean tilings on the plane are given in Fig. 1. These are all the Archimedean tilings on the plane $\mathbb{R}^2$ (cf. \cite{GS1977}). All of these are vertex-transitive maps.}
\end{example}

\bigskip

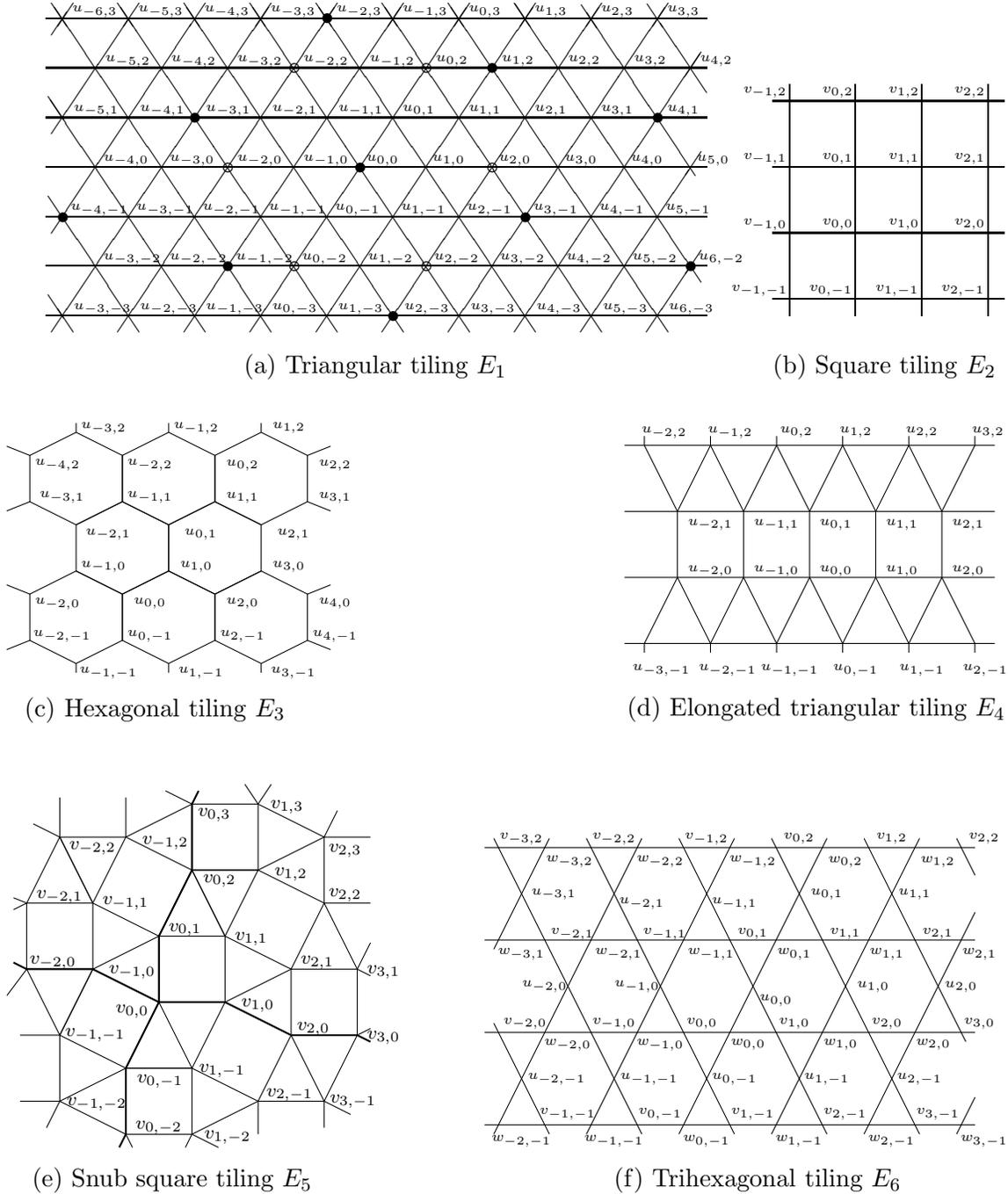
\setlength{\unitlength}{2.5mm}



\smallskip


\begin{figure}[ht!]
\tiny
\tikzstyle{ver}=[]
\tikzstyle{vert}=[circle, draw, fill=black!100, inner sep=0pt, minimum width=4pt]
\tikzstyle{vertex}=[circle, draw, fill=black!00, inner sep=0pt, minimum width=4pt]
\tikzstyle{edge} = [draw,thick,-]

%

\end{figure}



\begin{figure}[ht!]
\tiny
\tikzstyle{ver}=[]
\tikzstyle{vert}=[circle, draw, fill=black!100, inner sep=0pt, minimum width=4pt]
\tikzstyle{vertex}=[circle, draw, fill=black!00, inner sep=0pt, minimum width=4pt]
\tikzstyle{edge} = [draw,thick,-]
\centering

%

\begin{figure}[ht!]
\tiny
\tikzstyle{ver}=[]
\tikzstyle{vert}=[circle, draw, fill=black!100, inner sep=0pt, minimum width=4pt]
\tikzstyle{vertex}=[circle, draw, fill=black!00, inner sep=0pt, minimum width=4pt]
\tikzstyle{edge} = [draw,thick,-]
\centering

%

\end{figure}


\begin{figure}[ht!]
\tiny
\tikzstyle{ver}=[]
\tikzstyle{vert}=[circle, draw, fill=black!100, inner sep=0pt, minimum width=4pt]
\tikzstyle{vertex}=[circle, draw, fill=black!00, inner sep=0pt, minimum width=4pt]
\tikzstyle{edge} = [draw,thick,-]
\centering

%

\caption*{{\bf Figure 1: Archimedean tilings on the plane}}\label{fig:Archimedean}

\end{figure}



\begin{example} \label{exam:torus}
{\rm Five semi-equivelar maps on the torus are given in Fig. 2.  Maps $M_2$, $M_3$, $M_4$  were first constructed in \cite[Example 4.2]{DM2017}. We need these in the proof of Theorem \ref{theo:orbits-3&6}. }
\end{example}

\begin{figure}[ht!]
\tiny
\tikzstyle{ver}=[]
\tikzstyle{vert}=[circle, draw, fill=black!100, inner sep=0pt, minimum width=4pt]
\tikzstyle{vertex}=[circle, draw, fill=black!00, inner sep=0pt, minimum width=4pt]
\tikzstyle{edge} = [draw,thick,-]


\end{figure}

\begin{figure}[ht!]
\tiny
\tikzstyle{ver}=[]
\tikzstyle{vert}=[circle, draw, fill=black!100, inner sep=0pt, minimum width=4pt]
\tikzstyle{vertex}=[circle, draw, fill=black!00, inner sep=0pt, minimum width=4pt]
\tikzstyle{edge} = [draw,thick,-]
\centering

%

\end{figure}

\begin{figure}[ht!]
\tiny
\tikzstyle{ver}=[]
\tikzstyle{vert}=[circle, draw, fill=black!100, inner sep=0pt, minimum width=4pt]
\tikzstyle{vertex}=[circle, draw, fill=black!00, inner sep=0pt, minimum width=4pt]
\tikzstyle{edge} = [draw,thick,-]
\centering

%
\end{figure}

\section{Proofs of Theorems \ref{theo:plane} and \ref{theo:sem=archimedian}} \label{sec:proofs-1}

To prove Theorem \ref{theo:plane}, we need Lemmas \ref{lemma:E7}, \ref{lemma:E8}, \ref{lemma:E9}, \ref{lemma:E10} and \ref{lemma:E11}.

If $uv$ is an edge in a map $M$ then we say $u$ is {\em adjacent} to $v$ in $M$. For a map $M$ on a surface, $V(M)$ denotes the vertex set of $M$.

Consider the subset $U_0 =\{u_{7i+4j, j} \, : \, i, j\in\mathbb{Z}\}$ of $V(E_1)$ (see Fig. 1\,(a)). The set $U_0$ satisfies the following property.
\begin{align} \label{eq:first}
\mbox{\em Each $v$ in $V(E_1)\setminus U_0$ is adjacent in $E_1$ to a unique $u$ in $U_0$.}
\end{align}

Trivially, $V(E_1)$ satisfies \eqref{eq:first}. Here we prove (we need this lemma to prove Lemma \ref{lemma:E7}).


\begin{lemma} \label{lemma:U0}
 Up to an automorphism of $E_1$, $U_0$ is the unique proper subset of $V(E_1)$ which satisfies Property $\eqref{eq:first}$.
\end{lemma}

\begin{proof}
Let $V=V(E_1)$ be the vertex set of $E_1$. Let $U\subseteq V$ satisfy Property \eqref{eq:first}.

\medskip

\noindent {\bf Claim 1.} If there exist two vertices in $U$ which are adjacent in $E_1$ then $U=V$.

\smallskip

Suppose that $u, u_1 \in U$, where $uu_1$ is an edge in $E_1$. Let the link of $u$ in $E_1$ be $C_6(u_1, \dots, u_6)$. Then both $u$, $u_1$ are  adjacent to $u_2$ and $u, u_1\in U$. Since $U$ satisfies \eqref{eq:first}, $u_2\notin V\setminus U$. So, $u_2\in U$. Successively, we get $u_2, u_3, u_4, u_5, u_6\in U$. Thus, all the neighbours of $u$ are in $U$. Since $E_1$ is connected, this implies that $U=V(E_1)$. This proves Claim 1.

\smallskip

Now, assume that $U\subsetneqq  V(E_1)$. Then $W = V\setminus U \neq\emptyset$ and hence there exist $v\in W$ and $u\in U$ such that $u$ and $v$ are adjacent in $E_1$. For $x\in W$, let $N(x)$ be the unique vertex $y\in U$ so that $xy$ is an edge in $E_1$. Then, $x \mapsto N(x)$ defines a mapping $N : W \to U$.

Up to an automorphism of $E_1$, assume that $u_{1,0}\in W$ and $u_{0,0}\in U$.
Since $U$ satisfies \eqref{eq:first}, all the neighbours of $u_{1,0}$ other than $u_{0,0}$ are in $W$. So, $\{u_{0,1}, u_{1,1}, u_{2,0}, u_{2,-1}, u_{1,-1}\}\subseteq W$.
Therefore, $N(u_{2,0})\in \{u_{2,1}, u_{3,0}, u_{3,-1}\}$.
If $N(u_{2,0}) = u_{3,0}$ then other neighbours of $u_{2,0}$ are in $W$ and hence $\{u_{2,1}, u_{3,-1}\}\subseteq W$. This implies $N(u_{1,1}) = u_{0,2}$ or $u_{1,2}$.
In the first case, $u_{0,1}$ has two neighbours $u_{0,0}, u_{0,2}$ in $U$, a contradiction.
In the second case, $u_{2,1}$ has two neighbours $u_{3,0}, u_{1,2}$ in $U$, a contradiction.
Thus, $N(u_{2,0}) \neq u_{3,0}$. Therefore, $N(u_{2,0}) = u_{2,1}$ or $u_{3,-1}$.
Up to an automorphism of $E_1$, we can assume that $N(u_{2,0}) = u_{3,-1}$.
By, the similar arguments, $N(u_{2,-2}) = u_{3,-2}$ or $u_{2,-3}$.
Since $u_{3,-1}\in U$ is adjacent to $u_{2,-1}\in W$, $u_{3,-2}\not\in U$. Thus, $N(u_{2,-2}) = u_{2,-3}$. Similarly,
$N(u_{0,-2}) = u_{-1,-2}$, $N(u_{-2,0}) = u_{-3,1}$, $N(u_{-2,2}) = u_{-2,3}$  and $N(u_{0,2}) = u_{1,2}$. Thus
\begin{align*}
u_{3,-1}, u_{2,-3}, u_{-1,-2}, u_{-3,1}, u_{-2,3}, u_{1,2} \in U.
\end{align*}

Let $A:= \{(3,-1), (2,-3), (-1,-2), (-3,1), (-2,3), (1,2)\}\subseteq \mathbb{Z}\times \mathbb{Z}$. For $(k,\ell)\in \mathbb{Z}\times \mathbb{Z}$, let $V_{k,\ell} := \{u_{k+p,\ell+q} \, : \, (p, q)\in A\}$. (There are 18 vertices at a distance 3 in the graph $E_1$ from $u_{k,\ell}$. $V_{k,\ell}$ consists of 6 out of these 18.)
Clearly, $u_{p,q}\in V_{k,\ell}$ if and only if $u_{k,\ell}\in V_{p,q}$.

\medskip

\noindent {\bf Claim 2.} If $u_{k,\ell}\in U$ and $V_{k,\ell}\cap U\neq\emptyset$ then $V_{k,\ell}\subseteq U$.

\smallskip

Since $U\neq V(E_1)$ satisfies Property \eqref{eq:first} and $u_{k,\ell}\in U$, it follows that all the vertices of $E_1$ at a distance 1 or 2 from $u_{k,\ell}$ in the graph $E_1$ are in $W$.
Assume, without loss, that $u_{k+3, \ell-1}\in V_{k,\ell}\cap U$. (The other five cases are similar.) Then, by the same arguments as above, $N(u_{k+2,\ell-2}) = u_{k+2,\ell-3}$, $N(u_{k+0,\ell-2}) = u_{k-1,\ell -2}$, $N(u_{k-2,\ell}) = u_{k-3,\ell+1}$, $N(u_{k-2,\ell+2}) = u_{k-2,\ell+3}$  and $N(u_{k,\ell+2}) = u_{k+1,\ell+2}$. These imply,
$u_{k+2,\ell-3}$, $u_{k-1,\ell -2}$, $u_{k-3,\ell+1}$, $u_{k-2,\ell+3}$, $u_{k+1,\ell+2}\in U$. This proves Claim 2.

\smallskip

Let $V_0 := \{u_{0,0}\}$ and $V_n := \bigcup_{u_{k,\ell}\in V_{n-1}}V_{k,\ell}$ for $n\geq 1$.

\medskip

\noindent {\bf Claim 3.} $V_{n} \subseteq U$ for $n\geq 0$.

\smallskip

We prove the claim by induction on $n$.
From above, $V_0, V_1\subseteq U$. Assume that $n\geq 2$ and the claim is true for all $m <  n$. So, $V_{n-2}, V_{n-1} \subseteq U$. Let $u_{k,\ell}\in V_{n-1}$.
From the definition of $V_{n-1}$, there exists $u_{p,q}\in V_{n-2}$ such that $u_{k,\ell}\in V_{p,q}$. Then $u_{p,q}\in V_{k, \ell}$. So, $u_{p,q}\in V_{k, \ell}\cap V_{n-2} \subseteq V_{k, \ell}\cap U$. Thus, $V_{k, \ell}\cap U\neq\emptyset$. Also, $u_{k,\ell}\in V_{n-1} \subseteq U$. Therefore, by Claim 2, $V_{k, \ell}\subseteq U$. Since $u_{k,\ell}$ is an arbitrary vertex in $V_{n-1}$, it follows that $V_n = \bigcup_{u_{k,\ell}\in V_{n-1}}V_{k,\ell}\subseteq U$. The claim now follows by induction.

\smallskip

Consider the graph $G$ whose vertex set is $U_0$. Two vertices in $G$ are adjacent if their distance in $E_1$ is 3. Then the set of neighbours of a vertex $u_{k, \ell}$ is the set $V_{k,\ell}$ of six vertices. Thus, $G$ is a 6-regular graph.

\medskip

\noindent {\bf Claim 4.} $G$ is connected.

\smallskip

Observe that $u_{7i,0}\mbox{-}u_{7i+3,-1}\mbox{-}u_{7i+4,1}\mbox{-}u_{7i+7,0}$ is a path (of length 3) in $G$ from $u_{7i,0}$ to $u_{7(i+1),0}$. This implies that $u_{7h,0}$ and $u_{7i,0}$ are connected by a path in $G$, for $h\neq i\in \mathbb{Z}$. In particular, $u_{7i,0}$ is connected by a path in $G$ to $u_{0,0}$, for all $i\in \mathbb{Z}$.

From the definition of $G$, $u_{7i+4j, j}u_{7i+4j+3, j-1}$ is an edge in $G$ for all $i, j$.
This implies that, for any vertex $u_{7i+4j, j}$ of $G$, there exists a path of length $|j|$ from $u_{7i+4j, j}$ in $G$ to a vertex of the form $u_{7k, 0}$ for some $k\in \mathbb{Z}$. Therefore, by the conclusion in the previous paragraph, any $u_{7i+4j, j}$ is connected by a path in $G$ to $u_{0,0}$. Claim 4 follows from this.

\medskip

\noindent {\bf Claim 5.} $\bigcup_{n\geq 0}V_{n} = U_0$.

\smallskip

Since each element in $A$ is of the form $(7i+4j,j)$, it follows that $V_{k,\ell}\subseteq U_0$ for $u_{k,\ell}\in U_0$. Thus, $\bigcup_{n\geq 0}V_n \subseteq U_0$.  Let $H$ be the induced graph $G[\cup_{n\geq 0}V_{n}]$. Clearly, the set of neighbours of a vertex $u_{k, \ell}$ in $H$ is $V_{k,\ell}$. Hence $H$ is also 6-regular.
Since $G$ is connected and both $G$ and the subgraph $H$ are 6-regular, it follows that $H=G$. This implies that $\cup_{n\geq 0}V_{n} = U_0$. This proves Claim 5.

\smallskip

From Claims 3 and 5, $U_0 = \bigcup_{n=0}^{\infty}V_n \subseteq U$.
Suppose $U\neq U_0$. Let $u\in U\setminus U_0$. Since $U_0$ satisfies \eqref{eq:first}, there exists $v\in U_0$ such that $uv$ is an edge in $E_1$. Then, $u, v \in U$ and $uv$ is an edge in $E_1$. Therefore, by Claim 1, $U=V$. This is not possible by the assumption. Thus, $U = U_0$. This completes the proof.
\end{proof}

\begin{lemma} \label{lemma:E7}
Up to isomorphism, $E_7$ (given in Example $\ref{exam:plane}$) is the unique semi-equivelar map of type $[3^4, 6^1]$ on the plane. \end{lemma}

\begin{proof}
Let $Y$ be a semi-equivelar map of type $[3^4, 6^1]$ on the plane. For each 6-face $\sigma$ (face whose boundary is a 6-gon) of $Y$, choose a vertex $v_{\sigma}$ at the interior of $\sigma$ and join $v_{\sigma}$ by edges to all the six vertices of $\sigma$. Let $\widehat{Y}$ be the new map obtained from $Y$ by this way. Then $\deg_{\widehat{Y}}(v_{\sigma}) =6$ for each 6-face $\sigma$ of $Y$. Now, if $u$ is a vertex of $Y$ then $u$ is in a unique 6-face and $\deg_Y(u)=5$. These imply, that $\deg_{\widehat{Y}}(u) = 5+1=6$.
Clearly, each face in $\widehat{Y}$ is a triangle. Thus, $\widehat{Y}$ is a semi-equivelar map of type $[3^6]$. Hence, by \cite[Lemma 3.2]{DU2005}, $\widehat{Y} \cong E_1$. So, we can assume that $\widehat{Y} = E_1$.
Let $U = \{v_{\sigma} \, : \sigma$ is a 6-face of $Y\}$. Then, $V = V(\widehat{Y}) = U\sqcup V(Y)$. Since each $u\in V(Y)$ is in unique 6-face $\sigma$ of $Y$, each vertex $u\in V(Y)$ is adjacent in $\widehat{Y}$ to a unique vertex $v_{\sigma}$ in $U$. Thus, the subset $U$ of $V(\widehat{Y})$ satisfies Property \eqref{eq:first}.

On the other hand, let $U\subsetneqq V(E_1)$ satisfy Property \eqref{eq:first}. Let $N: V(E_1)\setminus U \to  U$ be the mapping as in the proof of Lemma \ref{lemma:U0}. Let $Z$ be the map obtained from $E_1$ by removing all the vertices in $U$ and all the edges through vertices in $U$. Now, let $u$ be a vertex in $U$. Then the neighbours of $u$ in $E_1$ form a 6-face (the 6-cycle ${\rm lk}_{E_1}(u)$) in $Z$. Since $U$ satisfies Property \eqref{eq:first}, it follows that each vertex $v$ in $Z$ is on a unique 6-face
(the 6-cycle ${\rm lk}_{E_1}(N(v)$). Also, the edges through $v$ in $Z$ are those in $E_1$ except the edge $vN(v)$. So, number of edges through $v$ in $Z$ is 5. These imply that $Z$ is a semi-equivelar map of type $[3^4, 6^1]$ on the plane.

Thus, there is a one to one correspondence between semi-equivelar maps of type $[3^4, 6^1]$ on the plane and the sets $U\subsetneqq V(E_1)$ which satisfy Property \eqref{eq:first}.  The result now follows by Lemma \ref{lemma:U0}.
\end{proof}

\begin{remark} \label{remark:U0}
{\rm If we take $u_{0,0}=(0,0)$ then $V(E_1)$ is a plane lattice in which $U_0$ is a sublattice of index 7. The mapping $N : V(E_1)\setminus U_0 \to U_0$ is a 6 to 1 map. If we consider the equivalence relation on $V(E_1)$ as $x\sim y$ if $x-y\in U_0$ then the quotient complex $E_1/\sim$ is the unique 7-vertex triangulation of the torus. In the proof of Lemma \ref{lemma:U0}, if we choose $N(u_{2,0}) =u_{2,1}$ (in place of $u_{3,-1}$) then we get $U= \{u_{7i+2j,j} \, : \, i, j\in\mathbb{Z}\}$ ($=\overline{U}_0$ say) in place of $U_0$. Observe that $\overline{U}_0 = \beta(U_0)$, where $\beta$ is the automorphism of $E_1$ given by $\beta(u_{i,j}) = u_{i +j, -j}$. ($\beta$ is the orthogonal reflection w.\,r.\,t. the line joining $u_{0, 0}$ and $u_{1, 0}$.) It follows from the proof of Lemma \ref{lemma:U0} that $U_0$, $\overline{U}_0$ and $V(E_1)$ are the only sets which contain $u_{0, 0}$ and satisfy Property \eqref{eq:first}. In fact, there are exactly 14 proper subsets of $V(E_1)$ which satisfy Property \eqref{eq:first}. These are $U_0, \dots, U_6, \overline{U}_0, \dots, \overline{U}_6$, where $U_i := \alpha^i(U_0)$, $\overline{U}_i = \alpha(\overline{U}_i) = \beta(U_i)$ for $1\leq i\leq 6$, and $\alpha$ is the translation $x\mapsto x+ (u_{1,0}-u_{0, 0})$. Among these, $U_0$ and $\overline{U}_0$ are sublattices of $E_1$. Observe that the subset $U_0$ of $V(E_1)$ corresponds to $E_7$ and the subset $\overline{U}_0$ of $V(E_1)$  corresponds to the other snub hexagonal tiling $\beta(E_7)$.}
\end{remark}

\begin{lemma} \label{lemma:E8}
Let $E_8$ be as in Example $\ref{exam:plane}$.
Let $X$ be a semi-equivelar map on the plane. If the type of $X$ is $[3^1, 4^1, 6^1, 4^1]$ then $X\cong E_8$.
\end{lemma}

To prove Lemma \ref{lemma:E8}, we need the following two technical lemmas.

\begin{lemma} \label{lemma:disc3464-1}
Let $X$ be a map on the $2$-disk $\mathbb{D}^2$ whose faces are triangles, quadrangles and hexagons. For a vertex $u$ of $X$, let $n_3(u),$ $n_4(u)$ and $n_6(u)$ be the number of triangles, quadrangles and hexagons through $u$ respectively. Then $X$ can not satisfy both the following: {\rm (i)} $(n_3(u), n_4(u), n_6(u)) = (1, 2, 1)$ for each internal vertex $u$, {\rm (ii)} boundary $\partial \mathbb{D}^2$ has $2k+\ell$ vertices, among these $2k$ vertices are in one hexagon and one quadrangle, and $\ell$ vertices are in one hexagon only, for some $k > 0$ and $k-1 \le l \le k+2$.
\end{lemma}

\begin{proof}
Let $f_0, f_1$ and $f_2$ denote the number of vertices, edges and faces of $X$ respectively. For $i= 3, 4, 6$, let $n_i$ denote the total number of $i$-gons in $X$. Let there be $n$ internal vertices. Then $f_0=n+2k+\ell$ and
$f_1= (4\times n+ 3\times 2k+ 2\times \ell)/2$.

Suppose $X$ satisfies (i) and (ii). Then $n_3=n/3$, $n_4= (2n+ 2k)/4 = (n+k)/2$ and $n_6=(n+2k+\ell)/6$. Since $n_4\in\mathbb{Z}$, $n+k$ is even. Also $n_6\in\mathbb{Z}$ implies $n+2k+\ell$ is even and hence $k+\ell$ is even.
Since $k-1 \le \ell \le k+2$, we get that $\ell = k$ or $k+2$. Since $n+k = 2n_3 + (n_3+k)$, it follows that $n_3+k$ is also even, say $n_3+k = 2m$. Now, if $\ell = k+2$ then $n_6 =(n+2k+\ell)/6= (3(n_3+k)+2)/6= m+1/3\not\in\mathbb{Z}$, a contradiction. Thus, $\ell=k$. Then, $f_0 = n+3k$, $f_1= (4n+ 6k+ 2k)/2 = 2n+4k$ and
$f_2= n_3+n_4+n_6 = n/3 + (n+k)/2 + (n+2k+k)/6= n+k$. Then $f_0-f_1+f_2 = 0$. This is not possible since the Euler characteristic of $\mathbb{D}^2$ is 1. This completes the proof.
\end{proof}

\begin{lemma} \label{lemma:disc3464-2}
Let $X$ be a map on the $2$-disk $\mathbb{D}^2$ whose faces are triangles, quadrangles and hexagons. For a vertex $u$ of $X$, let $n_3(u),$ $n_4(u)$ and $n_6(u)$ be the number of triangles, quadrangles and hexagons through $u$ respectively. Then $X$ can not satisfy both the following: {\rm (i)} $(n_3(u), n_4(u), n_6(u)) = (1, 2, 1)$ for each internal vertex $u$, {\rm (ii)} boundary $\partial \mathbb{D}^2$ has $2k+\ell$ vertices, among these $2k$ vertices are in one quadrangle and one triangle and $\ell$ vertices are in two quadrangles and one triangle, for some $k > 0$ and $k-1 \leq \ell \leq k+1$.
\end{lemma}

\begin{proof}
Let $f_0, f_1$ and $f_2$ denote the number of vertices, edges and faces of $X$ respectively. For $i= 3, 4, 6$, let $n_i$ denote the total number of $i$-gons in $X$. Let there be $n$ internal vertices. Then $f_0=n+2k+\ell$ and $f_1= (4\times n+3\times 2k+ 4\times\ell)/2$.

Suppose $X$ satisfies (i) and (ii).
Then $n_6=n/6$, $n_4 = (2n+2k+2\ell)/4 = (12n_6+2k+2\ell)/4 = 3n_6+(k+\ell)/2$.
Thus, $k+\ell$ is even. Since $k-1 \le \ell \le k+1$, it follows that $\ell = k$.
Then, $f_0 = n+3k$, $f_1= (4n+6k+4k)/2= 2n+5k$ and
$f_2= n_3 + n_4 +n_6 = (n+3k)/3 + (2n+4k)/4+ n/6= n+2k$. Therefore, $f_0-f_1+f_2 = 0$. This is not possible since the Euler characteristic of $\mathbb{D}^2$ is 1. This completes the proof.
\end{proof}


\begin{proof}[Proof of Lemma \ref{lemma:E8}]
Let $a_{0,0}b_{0,0}$ be an edge of a $6$-gon $A_{0,0}$. Then $a_{0,0}b_{0,0}$ is also in a $4$-gon $B_{0,0}$.  Let the face-cycle $C(b_{0,0})$ at $b_{0,0}$ be $C_4(A_{0,0}, B_{0,0}, C_{0,0}, D_{0,0})$, where $A_{0,0} =$
$C_6(b_{0,0}, a_{0,0}, c_{0,1}, b_{0,1}, a_{0,1}$, $c_{0,0})$, $B_{0,0} =$ $C_4(b_{0,0}$, $a_{0,0}$, $b_{0,-1}$, $a_{0,-1})$, $C_{0,0} =$ $C_3(b_{0,0}$, $a_{0,-1}$, $c_{1,-1})$, $D_{0,0} =$ $C_4(b_{0,0}$, $c_{1,-1}$, $b_{1,-1}$, $c_{0,0})$,
for some vertices $c_{0,1}$, $b_{0,1}$, $a_{0,1}$, $c_{0,0}$, $b_{1,-1}$, $c_{1,-1}$, $a_{0,-1}$, $b_{0,-1}$, $c_{-1,0}$, $a_{-1,1}$ (see Fig. 3).
Then the other two faces through $c_{0,0}$ must be of the form $E_{0,0} := C_3(c_{0,0}$, $b_{1,-1}$, $a_{1,0})$ and $F_{0,0} := C_4(c_{0,0}, a_{1,0}, c_{1,1}, a_{0,1})$.
This implies that the faces through $a_{1,0}$ are $E_{0,0}$, $F_{0,0}$, $A_{1,0} := C_6(a_{1,0}, c_{1,1}, b_{1,1}, a_{1,1}, c_{1,0}, b_{1,0})$, $B_{1,0} :=$ $C_4(a_{1,0}$, $b_{1,0}$, $a_{1,-1}$, $b_{1,-1})$.
Then the other two faces through $b_{1,0}$ must be of the form $C_{1,0} :=$ $C_3(b_{1,0}, a_{1,-1}, c_{2,-1})$, $D_{1,0} :=$ $C_4(b_{1,0}, c_{2,-1}, b_{2,-1}, c_{1,0})$.
Continuing this way, we get the paths $P_0 := \cdots\mbox{-} b_{-1,0}\mbox{-} c_{-1,0}\mbox{-}a_{0,0}\mbox{-}b_{0,0}\mbox{-}c_{0,0}\mbox{-}a_{1,0}\mbox{-}b_{1,0} \mbox{-}c_{1,0}\mbox{-} a_{2,0}\mbox{-} \cdots$, $Q_0 := \cdots\mbox{-} c_{0,-1} \mbox{-}b_{0,-1}\mbox{-}$ $a_{0,-1}\mbox{-}c_{1,-1}\mbox{-}b_{1,-1}\mbox{-} a_{1,-1}\mbox{-}c_{2,-1}\mbox{-}b_{2,-1} \mbox{-} \cdots$,  $Q_1 := \cdots \mbox{-} b_{-1,1}\mbox{-} a_{-1,1}\mbox{-}c_{0,1}\mbox{-}b_{0,1}\mbox{-}a_{0,1}\mbox{-}c_{1,1}\mbox{-}b_{1,1}\mbox{-} a_{1,1}\mbox{-} \cdots$ and the faces $B_{i,0}:=C_4(a_{i,0}, b_{i,0}, a_{i,-1}, b_{i,-1})$, $C_{i,0}:=C_3(b_{i,0}, c_{i+1,-1}, a_{i,-1})$, $D_{i,0}:=$ $C_4(b_{i,0}$, $c_{i,0}$, $b_{i+1,-1}$, $c_{i+1,-1})$, $E_{0,0}:=C_3(c_{i,0}, a_{i+1,0}, b_{i+1,-1})$, $i \in \mathbb{Z}$, between $Q_0$ and $P_0$ and faces $A_{i,0}:=C_6(a_{i,0}, b_{i,0}, c_{i,0}, a_{i,1}, b_{i,1}, c_{i,1})$, $F_{i,0}:=C_4(c_{i,0},$ $a_{i+1,0}, c_{i+1,1}, a_{i,1})$, $i \in \mathbb{Z}$, between $P_0$ and $Q_1$ (see Fig. 3).


\begin{figure}[ht]
\tiny
\tikzstyle{ver}=[]
\tikzstyle{vert}=[circle, draw, fill=black!100, inner sep=0pt, minimum width=4pt]
\tikzstyle{vertex}=[circle, draw, fill=black!00, inner sep=0pt, minimum width=4pt]
\tikzstyle{edge} = [draw,thick,-]
\centering

\begin{tikzpicture}[scale=0.45]

\draw ({sqrt(3)}, 1) -- (0, 2) -- ({-sqrt(3)}, 1);
\draw  ({6+sqrt(3)}, 1) -- (6+0, 2) -- ({6-sqrt(3)}, 1);
\draw ({12+sqrt(3)}, 1) -- (12+0, 2) -- ({12-sqrt(3)}, 1);

\draw [fill = lightgray] ({-4.7+sqrt(3)}, 2.55) -- ({-sqrt(3)}, 1) -- ({-4.7+sqrt(3)}, 1);
\draw [fill = lightgray] ({-4.7+sqrt(3)}, 4) -- ({-4.7+sqrt(3)}, 2.55) -- ({-sqrt(3)}, 1) -- (0, 2) --  ({sqrt(3)}, 1) -- ({6-sqrt(3)}, 1) -- (6+0, 2) -- ({6+sqrt(3)}, 1) -- ({12-sqrt(3)}, 1) --(12+0, 2) -- ({12+sqrt(3)}, 1) -- ({16.7-sqrt(3)}, 2.55) -- ({16.7-sqrt(3)}, 4);
\draw [fill = lightgray] ({16.7-sqrt(3)}, 2.55)  -- ({12+sqrt(3)}, 1) -- ({16.7-sqrt(3)}, 1);

\draw ({-4.8+sqrt(3)}, 1) -- ({-sqrt(3)}, 1);
\draw ({sqrt(3)}, 1) -- ({6-sqrt(3)}, 1);
\draw ({6+sqrt(3)}, 1) -- ({12-sqrt(3)}, 1);
\draw ({12+sqrt(3)}, 1) -- ({16.8-sqrt(3)}, 1);

\draw ({-sqrt(3)}, 1) -- (-3+0, 2.6);
\draw (0, 2) -- (-1.25, 3.6);
\draw (0, 2) -- (1.3, 3.6);
\draw ({sqrt(3)}, 1) -- (3.05, 2.6);

\draw ({6-sqrt(3)}, 1) -- (3+0, 2.6);
\draw (6, 2) -- (4.75, 3.6);
\draw (6, 2) -- (7.3, 3.6);
\draw ({6+sqrt(3)}, 1) -- (9.05, 2.6);

\draw ({12-sqrt(3)}, 1) -- (9, 2.6);
\draw (12, 2) -- (10.75, 3.6);
\draw (12, 2) -- (13.3, 3.6);
\draw ({12+sqrt(3)}, 1) -- (15.05, 2.6);


\draw [fill=lightgray, xshift = 86, yshift = 130] (-6+0, -2) -- ({-6+sqrt(3)}, -1) -- ({-6+sqrt(3)}, 1) -- ({-6+sqrt(3)}, 1) -- (-6+0, 2);

\draw [fill=lightgray, xshift = 86, yshift = 130] ({sqrt(3)}, 1) -- (0, 2) -- ({-sqrt(3)}, 1) -- ({-sqrt(3)}, -1) -- (0, -2) -- ({sqrt(3)}, -1) -- ({sqrt(3)}, 1);

\draw [fill=lightgray, xshift = 86, yshift = 130] ({6+sqrt(3)}, 1) -- (6+0, 2) -- ({6-sqrt(3)}, 1) -- ({6-sqrt(3)}, -1) -- (6+0, -2) -- ({6+sqrt(3)}, -1) -- ({6+sqrt(3)}, 1);

\draw [fill=lightgray,xshift = 86, yshift = 130] (12+0, 2) -- ({12-sqrt(3)}, 1) -- ({12-sqrt(3)}, -1) -- (12+0, -2) ;

\draw [fill=lightgray, xshift = 86, yshift = 130] ({-6+sqrt(3)}, 1) -- ({-sqrt(3)}, 1)  -- ({-sqrt(3)}, -1)-- ({-6+sqrt(3)}, -1);

\draw [fill=lightgray, xshift = 86, yshift = 130] ({sqrt(3)}, 1) -- ({6-sqrt(3)}, 1) -- ({6-sqrt(3)}, -1) -- ({sqrt(3)}, -1);

\draw [fill=lightgray, xshift = 86, yshift = 130] ({6+sqrt(3)}, 1) -- ({12-sqrt(3)}, 1) -- ({12-sqrt(3)}, -1) -- ({6+sqrt(3)}, -1);


\draw [line width=0.3mm, xshift = 86, yshift = 130] ({-6+sqrt(3)}, -1) -- ({-6+sqrt(3)}, 1);
\draw [line width=0.1mm, xshift = 86, yshift = 130] ({-sqrt(3)}, -1) -- ({-sqrt(3)}, 1);
\draw [line width=0.3mm, xshift = 86, yshift = 130] ({sqrt(3)}, -1) -- ({sqrt(3)}, 1);
\draw [line width=0.1mm, xshift = 86, yshift = 130] ({6-sqrt(3)}, -1) -- ({6-sqrt(3)}, 1);
\draw [line width=0.3mm, xshift = 86, yshift = 130] ({6+sqrt(3)}, -1) -- ({6+sqrt(3)}, 1);
\draw [line width=0.1mm, xshift = 86, yshift = 130] ({12-sqrt(3)}, -1) -- ({12-sqrt(3)}, 1);


\draw [yshift = 260] ({-sqrt(3)}, -1) -- (0, -2) -- ({sqrt(3)}, -1);
\draw [yshift = 260] ({6-sqrt(3)}, -1) -- (6+0, -2) -- ({6+sqrt(3)}, -1);
\draw [yshift = 260] ({12-sqrt(3)}, -1) -- (12+0, -2) -- ({12+sqrt(3)}, -1);

\draw [yshift = 260] ({-4.8+sqrt(3)}, -1) -- ({-sqrt(3)}, -1);
\draw [yshift = 260] ({sqrt(3)}, -1) -- ({6-sqrt(3)}, -1);
\draw [yshift = 260] ({6+sqrt(3)}, -1) -- ({12-sqrt(3)}, -1);
\draw [yshift = 260] ({12+sqrt(3)}, -1) -- ({16.8-sqrt(3)}, -1);

\draw [xshift = 86, yshift = 130] ({-sqrt(3)}, 1) -- (-3+0, 2.6);
\draw [xshift = 86, yshift = 130](0, 2) -- (-1.25, 3.6);
\draw [xshift = 86, yshift = 130](0, 2) -- (1.3, 3.6);
\draw [xshift = 86, yshift = 130]({sqrt(3)}, 1) -- (3.05, 2.6);

\draw [xshift = 86, yshift = 130]({6-sqrt(3)}, 1) -- (3+0, 2.6);
\draw [xshift = 86, yshift = 130](6, 2) -- (4.75, 3.6);
\draw [xshift = 86, yshift = 130](6, 2) -- (7.3, 3.6);
\draw [xshift = 86, yshift = 130]({6+sqrt(3)}, 1) -- (9.05, 2.6);

\draw [xshift = 86, yshift = 130]({12-sqrt(3)}, 1) -- (9, 2.6);
\draw [xshift = 86, yshift = 130](12, 2) -- (10.75, 3.6);
\draw [yshift = 130](-3, 2) -- (-1.7, 3.6);
\draw [yshift = 130]({-3+sqrt(3)}, 1) -- (.05, 2.6);


\draw [yshift = 260] ({sqrt(3)}, 1) -- (0, 2) -- ({-sqrt(3)}, 1) -- ({-sqrt(3)}, -1) -- (0, -2) -- ({sqrt(3)}, -1) -- ({sqrt(3)}, 1);
\draw [yshift = 260] ({6+sqrt(3)}, 1) -- (6+0, 2) -- ({6-sqrt(3)}, 1) -- ({6-sqrt(3)}, -1) -- (6+0, -2) -- ({6+sqrt(3)}, -1) -- ({6+sqrt(3)}, 1);
\draw [yshift = 260]({12+sqrt(3)}, 1) -- (12+0, 2) -- ({12-sqrt(3)}, 1) -- ({12-sqrt(3)}, -1) -- (12+0, -2) -- ({12+sqrt(3)}, -1) -- ({12+sqrt(3)}, 1);

\draw [yshift = 260] ({-4.8+sqrt(3)}, 1) -- ({-sqrt(3)}, 1);
\draw [yshift = 260] ({-4.8+sqrt(3)}, -1) -- ({-sqrt(3)}, -1);
\draw [yshift = 260] ({sqrt(3)}, 1) -- ({6-sqrt(3)}, 1);
\draw [yshift = 260] ({sqrt(3)}, -1) -- ({6-sqrt(3)}, -1);
\draw [yshift = 260] ({6+sqrt(3)}, 1) -- ({12-sqrt(3)}, 1);
\draw [yshift = 260] ({6+sqrt(3)}, -1) -- ({12-sqrt(3)}, -1);
\draw [yshift = 260] ({12+sqrt(3)}, 1) -- ({16.8-sqrt(3)}, 1);
\draw [yshift = 260] ({12+sqrt(3)}, -1) -- ({16.8-sqrt(3)}, -1);


\node[ver] () at (1,.6){$a_{0,-1}$};
\node[ver] () at (0,1.3){$b_{0,-1}$};
\node[ver] () at (-1,.6){$c_{0,-1}$};

\node[ver] () at (1+6,.6){$a_{1,-1}$};
\node[ver] () at (0+6,1.3){$b_{1,-1}$};
\node[ver] () at (-1+6,.6){$c_{1,-1}$};

\node[ver] () at (1+12,.6){$a_{2,-1}$};
\node[ver] () at (0+12,1.3){$b_{2,-1}$};
\node[ver] () at (-1+12,.6){$c_{2,-1}$};
\node[ver] () at (0-3,-1.5+4.6){$b_{-1,0}$};
\node[ver] () at (.7-3,-.8+4.5){$c_{-1,0}$};
\node[ver] () at (.7-3,.6+4.5){$a_{-1,1}$};
\node[ver] () at (0-3,1.3+4.6){$b_{-1,1}$};

\node[ver] () at (-1+3,-.8+4.5){$a_{0,0}$};
\node[ver] () at (0+3,-1.5+4.6){$b_{0,0}$};
\node[ver] () at (1+3,-.8+4.5){$c_{0,0}$};
\node[ver] () at (1+3,.6+4.5){$a_{0,1}$};
\node[ver] () at (0+3,1.3+4.6){$b_{0,1}$};
\node[ver] () at (-1+3,.6+4.5){$c_{0,1}$};

\node[ver] () at (-1+9,-.8+4.5){$a_{1,0}$};
\node[ver] () at (0+9,-1.5+4.6){$b_{1,0}$};
\node[ver] () at (1+9,-.8+4.5){$c_{1,0}$};
\node[ver] () at (1+9,.6+4.5){$a_{1,1}$};
\node[ver] () at (0+9,1.3+4.6){$b_{1,1}$};
\node[ver] () at (-1+9,.6+4.5){$c_{1,1}$};

\node[ver] () at (-1+15,-.8+4.5){$a_{2,0}$};
\node[ver] () at (0+15,-1.5+4.6){$b_{2,0}$};
\node[ver] () at (0+15,1.3+4.6){$b_{2,1}$};
\node[ver] () at (-1+15,.6+4.5){$c_{2,1}$};
\node[ver] () at (-1,-.8+9){$a_{-1,2}$};
\node[ver] () at (0,-1.5+9.1){$b_{-1,2}$};
\node[ver] () at (1,-.8+9){$c_{-1,2}$};
\node[ver] () at (1,.6+9){$a_{-1,3}$};
\node[ver] () at (0,1.3+9.1){$b_{-1,3}$};
\node[ver] () at (-1,.6+9){$c_{-1,3}$};

\node[ver] () at (-1+6,-.8+9){$a_{0,2}$};
\node[ver] () at (0+6,-1.5+9.1){$b_{0,2}$};
\node[ver] () at (1+6,-.8+9){$c_{0,2}$};
\node[ver] () at (1+6,.6+9){$a_{0,3}$};
\node[ver] () at (0+6,1.3+9.1){$b_{0,3}$};
\node[ver] () at (-1+6,.6+9){$c_{0,3}$};

\node[ver] () at (-1+12,-.8+9){$a_{1,2}$};
\node[ver] () at (0+12,-1.5+9.1){$b_{1,2}$};
\node[ver] () at (1+12,-.8+9){$c_{1,2}$};
\node[ver] () at (1+12,.6+9){$a_{1,3}$};
\node[ver] () at (0+12,1.3+9.1){$b_{1,3}$};
\node[ver] () at (-1+12,.6+9){$c_{1,3}$};

\node[ver] () at (16,10){$Q_2$};
\node[ver] () at (16,8){$P_1$};
\node[ver] () at (16,6.5){$Q_1$};
\node[ver] () at (16,2.5){$P_0$};
\node[ver] () at (16,1){$Q_0$};

\node[ver] () at (1.5,2.4){$B_{0,0}$};
\node[ver] () at (4.5,2.4){$D_{0,0}$};
\node[ver] () at (7.5,2.4){$B_{1,0}$};
\node[ver] () at (10.5,2.4){$D_{1,0}$};

\node[ver] () at (3,1.5){$C_{0,0}$};
\node[ver] () at (9,1.5){$C_{1,0}$};

\node[ver] () at (0,4.5){$F_{-1,0}$};
\node[ver] () at (3,4.5){$A_{0,0}$};
\node[ver] () at (6,4.5){$F_{0,0}$};
\node[ver] () at (9,4.5){$A_{1,0}$};
\node[ver] () at (12,4.5){$F_{1,0}$};

\node[ver] () at (6,3){$E_{0,0}$};
\node[ver] () at (12,3){$E_{1,0}$};

\node[ver] () at (1.5,6.7){$D_{-1,1}$};
\node[ver] () at (3,7.7){$E_{-1,1}$};
\node[ver] () at (4.5,6.7){$B_{0,1}$};
\node[ver] () at (6,6){$C_{0,1}$};
\node[ver] () at (7.5,6.7){$D_{0,1}$};
\node[ver] () at (9,7.7){$E_{0,1}$};
\node[ver] () at (10.5,6.7){$B_{1,1}$};
\node[ver] () at (12,6){$C_{1,1}$};

\node[ver] () at (5.2, -1){\normalsize \bf Figure 3: Part of {\boldmath $X$} of Lemma \ref{lemma:E8}};

\end{tikzpicture}

\end{figure}



\noindent {\bf Claim.} $P_0$, $Q_0$, $Q_1$ are infinite paths.

\smallskip

If not then we get a cycle  (as a subgraph of $P_0$). This gives a map $M$ on the disc $\mathbb{D}^2$ in which each internal vertex is in one 6-gon, two 4-gons, one 3-gon and the boundary vertices are either (i) on  4-gons and 6-gons or (ii) on 3-gons and 4-gons. In the first case, $M$ would be like one in Fig. 4. But, this is not possible by Lemma \ref{lemma:disc3464-1}.
In the second case, $M$ would be like one in Fig. 5. This is also not possible by Lemma \ref{lemma:disc3464-2}. Thus, $P_0$ is an infinite path. Similarly, $Q_0$ and $Q_1$ are infinite paths. This proves the claim.

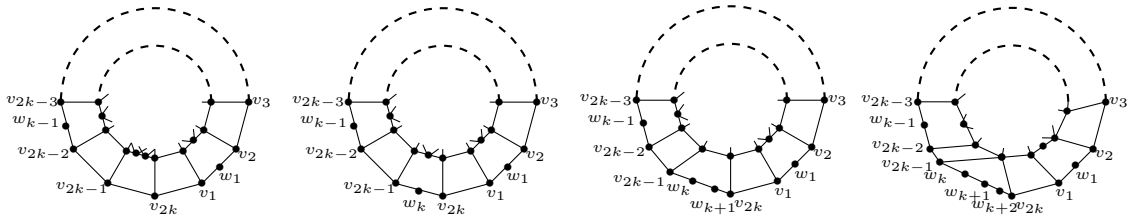
\begin{figure}[ht]
\tiny
\tikzstyle{ver}=[]
\tikzstyle{vert}=[circle, draw, fill=black!100, inner sep=0pt, minimum width=4pt]
\tikzstyle{vertex}=[circle, draw, fill=black!00, inner sep=0pt, minimum width=4pt]
\tikzstyle{edge} = [draw,thick,-]
\centering

\begin{tikzpicture}[scale=0.25]

\draw [thick,dashed] (5,0) arc (0:180:5cm);\draw [thick,dashed] (3,0) arc (0:180:3cm);

\draw (180:5 cm) -- (210:5 cm) -- (240:5 cm) -- (270:5 cm) -- (300:5 cm)-- (330:5 cm)-- (0:5 cm);

\draw (180:3 cm) -- (210:3 cm) -- (240:3 cm) -- (270:3 cm) -- (300:3 cm)-- (330:3 cm)-- (0:3 cm);

\draw (180:5 cm) -- (180:3 cm);
\draw (210:5 cm) -- (210:3 cm);
\draw (240:5 cm) -- (240:3 cm);
\draw (270:5 cm) -- (270:3 cm);
\draw (300:5 cm) -- (300:3 cm);
\draw (330:5 cm) -- (330:3 cm);
\draw (0:5 cm) -- (0:3 cm);

\draw (180:3 cm) -- (170:2.5 cm);
\draw (195:3 cm) -- (185:2.5 cm);\draw (195:3 cm) -- (205:2.5 cm);
\draw (210:3 cm) -- (210:2.5 cm);
\draw (240:3 cm) -- (240:2.5 cm);
\draw (270:3 cm) -- (270:2.5 cm);
\draw (300:3 cm) -- (300:2.5 cm);
\draw (315:3 cm) -- (305:2.5 cm);\draw (315:3 cm) -- (325:2.5 cm);
\draw (330:3 cm) -- (330:2.5 cm);
\draw (360:3 cm) -- (360:2.5 cm);
\draw (260:3 cm) -- (250:2.5 cm);\draw (260:3 cm) -- (270:2.5 cm);
\draw (250:3 cm) -- (260:2.5 cm);\draw (250:3 cm) -- (240:2.5 cm);

\node[ver] () at (180:5 cm){$\bullet$};
\node[ver] () at (180:3 cm){$\bullet$};
\node[ver] () at (210:5 cm){$\bullet$};
\node[ver] () at (210:3 cm){$\bullet$};
\node[ver] () at (240:5 cm){$\bullet$};
\node[ver] () at (240:3 cm){$\bullet$};
\node[ver] () at (270:5 cm){$\bullet$};
\node[ver] () at (270:3 cm){$\bullet$};
\node[ver] () at (300:5 cm){$\bullet$};
\node[ver] () at (300:3 cm){$\bullet$};
\node[ver] () at (330:5 cm){$\bullet$};
\node[ver] () at (330:3 cm){$\bullet$};
\node[ver] () at (0:5 cm){$\bullet$};
\node[ver] () at (0:3 cm){$\bullet$};

\node[ver] () at (195:4.9 cm){$\bullet$};
\node[ver] () at (195:2.9 cm){$\bullet$};

\node[ver] () at (250:2.9 cm){$\bullet$};
\node[ver] () at (260:2.9 cm){$\bullet$};

\node[ver] () at (315:4.9 cm){$\bullet$};
\node[ver] () at (315:2.9 cm){$\bullet$};

\node[ver] () at (190:6.4 cm){$w_{k-1}$};
\node[ver] () at (315:5.7 cm){$w_{1}$};

\node[ver] () at (180:6.5 cm){$v_{2k-3}$};
\node[ver] () at (205:6.5 cm){$v_{2k-2}$};
\node[ver] () at (230:6 cm){$v_{2k-1}$};
\node[ver] () at (275:5.7 cm){$v_{2k}$};
\node[ver] () at (300:5.7 cm){$v_{1}$};
\node[ver] () at (330:5.7 cm){$v_{2}$};
\node[ver] () at (0:5.7 cm){$v_{3}$};


\end{tikzpicture}
\begin{tikzpicture}[scale=0.25]

\draw [thick,dashed] (5,0) arc (0:180:5cm);\draw [thick,dashed] (3,0) arc (0:180:3cm);

\draw (180:5 cm) -- (210:5 cm) -- (240:5 cm) -- (270:5 cm) -- (300:5 cm)-- (330:5 cm)-- (0:5 cm);

\draw (180:3 cm) -- (210:3 cm) -- (240:3 cm) -- (270:3 cm) -- (300:3 cm)-- (330:3 cm)-- (0:3 cm);

\draw (180:5 cm) -- (180:3 cm);
\draw (210:5 cm) -- (210:3 cm);
\draw (240:5 cm) -- (240:3 cm);
\draw (270:5 cm) -- (270:3 cm);
\draw (300:5 cm) -- (300:3 cm);
\draw (330:5 cm) -- (330:3 cm);
\draw (0:5 cm) -- (0:3 cm);

\draw (180:3 cm) -- (170:2.5 cm);
\draw (195:3 cm) -- (185:2.5 cm);\draw (195:3 cm) -- (205:2.5 cm);
\draw (210:3 cm) -- (210:2.5 cm);
\draw (240:3 cm) -- (240:2.5 cm);
\draw (270:3 cm) -- (270:2.5 cm);
\draw (300:3 cm) -- (300:2.5 cm);
\draw (315:3 cm) -- (305:2.5 cm);\draw (315:3 cm) -- (325:2.5 cm);
\draw (330:3 cm) -- (330:2.5 cm);
\draw (360:3 cm) -- (360:2.5 cm);
\draw (255:3 cm) -- (265:2.5 cm);\draw (255:3 cm) -- (245:2.5 cm);

\node[ver] () at (180:5 cm){$\bullet$};
\node[ver] () at (180:3 cm){$\bullet$};
\node[ver] () at (210:5 cm){$\bullet$};
\node[ver] () at (210:3 cm){$\bullet$};
\node[ver] () at (240:5 cm){$\bullet$};
\node[ver] () at (240:3 cm){$\bullet$};
\node[ver] () at (270:5 cm){$\bullet$};
\node[ver] () at (270:3 cm){$\bullet$};
\node[ver] () at (300:5 cm){$\bullet$};
\node[ver] () at (300:3 cm){$\bullet$};
\node[ver] () at (330:5 cm){$\bullet$};
\node[ver] () at (330:3 cm){$\bullet$};
\node[ver] () at (0:5 cm){$\bullet$};
\node[ver] () at (0:3 cm){$\bullet$};

\node[ver] () at (195:4.9 cm){$\bullet$};
\node[ver] () at (195:2.9 cm){$\bullet$};

\node[ver] () at (255:2.9 cm){$\bullet$};
\node[ver] () at (255:4.9 cm){$\bullet$};

\node[ver] () at (315:4.9 cm){$\bullet$};
\node[ver] () at (315:2.9 cm){$\bullet$};

\node[ver] () at (190:6.4 cm){$w_{k-1}$};
\node[ver] () at (255:5.7 cm){$w_{k}$};
\node[ver] () at (315:5.7 cm){$w_{1}$};

\node[ver] () at (180:6.6 cm){$v_{2k-3}$};
\node[ver] () at (205:6.5 cm){$v_{2k-2}$};
\node[ver] () at (230:6 cm){$v_{2k-1}$};
\node[ver] () at (275:5.7 cm){$v_{2k}$};
\node[ver] () at (300:5.7 cm){$v_{1}$};
\node[ver] () at (330:5.7 cm){$v_{2}$};
\node[ver] () at (0:5.7 cm){$v_{3}$};


\end{tikzpicture}
\begin{tikzpicture}[scale=0.25]

\draw [thick,dashed] (5,0) arc (0:180:5cm);\draw [thick,dashed] (3,0) arc (0:180:3cm);

\draw (180:5 cm) -- (210:5 cm) -- (230:5 cm) -- (270:5 cm) -- (300:5 cm)-- (330:5 cm)-- (0:5 cm);

\draw (180:3 cm) -- (210:3 cm) -- (240:3 cm) -- (270:3 cm) -- (300:3 cm)-- (330:3 cm)-- (0:3 cm);

\draw (180:5 cm) -- (180:3 cm);
\draw (210:5 cm) -- (210:3 cm);
\draw (230:5 cm) -- (240:3 cm);
\draw (270:5 cm) -- (270:3 cm);
\draw (300:5 cm) -- (300:3 cm);
\draw (330:5 cm) -- (330:3 cm);
\draw (0:5 cm) -- (0:3 cm);

\node[ver] () at (180:5 cm){$\bullet$};
\node[ver] () at (180:3 cm){$\bullet$};
\node[ver] () at (210:5 cm){$\bullet$};
\node[ver] () at (210:3 cm){$\bullet$};
\node[ver] () at (230:5 cm){$\bullet$};
\node[ver] () at (240:3 cm){$\bullet$};
\node[ver] () at (270:5 cm){$\bullet$};
\node[ver] () at (270:3 cm){$\bullet$};
\node[ver] () at (300:5 cm){$\bullet$};
\node[ver] () at (300:3 cm){$\bullet$};
\node[ver] () at (330:5 cm){$\bullet$};
\node[ver] () at (330:3 cm){$\bullet$};
\node[ver] () at (0:5 cm){$\bullet$};
\node[ver] () at (0:3 cm){$\bullet$};

\node[ver] () at (195:4.8 cm){$\bullet$};
\node[ver] () at (195:2.9 cm){$\bullet$};
\node[ver] () at (245:4.75 cm){$\bullet$};
\node[ver] () at (260:4.8 cm){$\bullet$};
\node[ver] () at (315:4.85 cm){$\bullet$};
\node[ver] () at (315:2.9 cm){$\bullet$};

\draw (180:3 cm) -- (170:2.5 cm);
\draw (195:3 cm) -- (185:2.5 cm);\draw (195:3 cm) -- (205:2.5 cm);
\draw (210:3 cm) -- (210:2.5 cm);
\draw (240:3 cm) -- (240:2.5 cm);
\draw (270:3 cm) -- (270:2.5 cm);
\draw (300:3 cm) -- (300:2.5 cm);
\draw (315:3 cm) -- (305:2.5 cm);\draw (315:3 cm) -- (325:2.5 cm);
\draw (330:3 cm) -- (330:2.5 cm);
\draw (360:3 cm) -- (360:2.5 cm);

\node[ver] () at (190:6.4 cm){$w_{k-1}$};
\node[ver] () at (240:5.5 cm){$w_{k}$};
\node[ver] () at (261:5.8 cm){$w_{k+1}$};
\node[ver] () at (315:5.7 cm){$w_{1}$};

\node[ver] () at (180:6.6 cm){$v_{2k-3}$};
\node[ver] () at (205:6.5 cm){$v_{2k-2}$};
\node[ver] () at (222:6.5 cm){$v_{2k-1}$};
\node[ver] () at (280:5.7 cm){$v_{2k}$};
\node[ver] () at (300:5.7 cm){$v_{1}$};
\node[ver] () at (330:5.7 cm){$v_{2}$};
\node[ver] () at (0:5.7 cm){$v_{3}$};


\end{tikzpicture}
\begin{tikzpicture}[scale=0.25]

\draw [thick,dashed] (5,0) arc (0:180:5cm);\draw [thick,dashed] (3,0) arc (0:180:3cm);

\draw (180:5 cm) -- (210:5 cm) -- (220:5 cm) -- (270:5 cm) -- (300:5 cm)-- (330:5 cm)-- (0:5 cm);

\draw (180:3 cm) -- (230:3 cm) -- (260:3 cm) -- (290:3 cm) -- (320:3 cm)-- (350:3 cm)-- (0:3 cm);

\draw (180:5 cm) -- (180:3 cm);
\draw (210:5 cm) -- (230:3 cm);
\draw (220:5 cm) -- (260:3 cm);
\draw (270:5 cm) -- (260:3 cm);
\draw (300:5 cm) -- (290:3 cm);
\draw (330:5 cm) -- (320:3 cm);
\draw (0:5 cm) -- (350:3 cm);

\draw (180:3 cm) -- (170:2.5 cm);
\draw (205:3 cm) -- (195:2.5 cm);\draw (205:3 cm) -- (215:2.5 cm);
\draw (230:3 cm) -- (230:2.5 cm);
\draw (260:3 cm) -- (260:2.5 cm);
\draw (290:3 cm) -- (290:2.5 cm);
\draw (320:3 cm) -- (310:2.5 cm);\draw (320:3 cm) -- (330:2.5 cm);
\draw (350:3 cm) -- (350:2.5 cm);

\node[ver] () at (180:5 cm){$\bullet$};
\node[ver] () at (180:3 cm){$\bullet$};
\node[ver] () at (210:5 cm){$\bullet$};
\node[ver] () at (230:3 cm){$\bullet$};
\node[ver] () at (220:5 cm){$\bullet$};
\node[ver] () at (260:3 cm){$\bullet$};
\node[ver] () at (270:5 cm){$\bullet$};
\node[ver] () at (260:3 cm){$\bullet$};
\node[ver] () at (300:5 cm){$\bullet$};
\node[ver] () at (290:3 cm){$\bullet$};
\node[ver] () at (330:5 cm){$\bullet$};
\node[ver] () at (320:3 cm){$\bullet$};
\node[ver] () at (0:5 cm){$\bullet$};
\node[ver] () at (350:3 cm){$\bullet$};

\node[ver] () at (195:4.8 cm){$\bullet$};
\node[ver] () at (205:2.8 cm){$\bullet$};

\node[ver] () at (237.5:4.6 cm){$\bullet$};
\node[ver] () at (252:4.6 cm){$\bullet$};
\node[ver] () at (262.5:4.8 cm){$\bullet$};

\node[ver] () at (315:4.8 cm){$\bullet$};
\node[ver] () at (305:2.9 cm){$\bullet$};


\node[ver] () at (190:6.4 cm){$w_{k-1}$};
\node[ver] () at (225:5.7 cm){$w_{k}$};
\node[ver] () at (245:5.5 cm){$w_{k+1}$};
\node[ver] () at (261:5.7 cm){$w_{k+2}$};
\node[ver] () at (315:5.7 cm){$w_{1}$};

\node[ver] () at (180:6.6 cm){$v_{2k-3}$};
\node[ver] () at (202:6.5 cm){$v_{2k-2}$};
\node[ver] () at (212:6.5 cm){$v_{2k-1}$};
\node[ver] () at (280:5.7 cm){$v_{2k}$};
\node[ver] () at (300:5.7 cm){$v_{1}$};
\node[ver] () at (330:5.7 cm){$v_{2}$};
\node[ver] () at (0:5.7 cm){$v_{3}$};


\end{tikzpicture}

\vspace{-2mm}
\caption*{{\bf Figure 4: Some maps on the 2-disc}}
\end{figure}


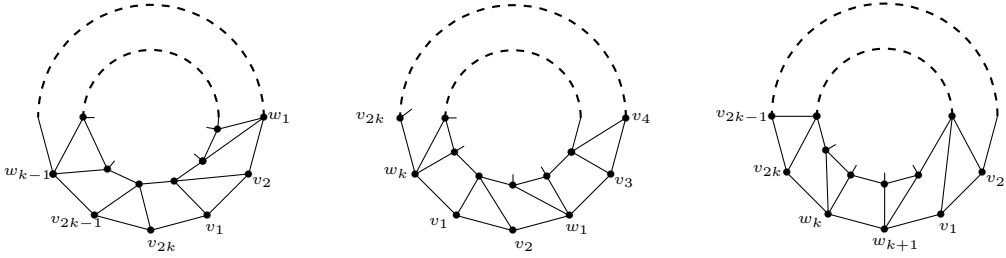
\begin{figure}[ht]
\tiny
\tikzstyle{ver}=[]
\tikzstyle{vert}=[circle, draw, fill=black!100, inner sep=0pt, minimum width=4pt]
\tikzstyle{vertex}=[circle, draw, fill=black!00, inner sep=0pt, minimum width=4pt]
\tikzstyle{edge} = [draw,thick,-]
\centering

\begin{tikzpicture}[scale=0.3]

\draw [thick,dashed] (5,0) arc (0:180:5cm);\draw [thick,dashed] (3,0) arc (0:180:3cm);

\draw (180:5 cm) -- (210:5 cm) -- (240:5 cm) -- (270:5 cm) -- (300:5 cm)-- (330:5 cm)-- (0:5 cm);

\draw (180:3 cm) -- (230:3 cm) -- (260:3 cm) -- (290:3 cm) -- (320:3 cm)-- (350:3 cm)-- (0:3 cm);

\draw (180:3 cm) -- (210:5 cm);
\draw (230:3 cm) -- (210:5 cm);
\draw (260:3 cm) -- (240:5 cm);
\draw (260:3 cm) -- (270:5 cm);
\draw (290:3 cm) -- (300:5 cm);
\draw (290:3 cm) -- (330:5 cm);
\draw (320:3 cm) -- (0:5 cm);
\draw (350:3 cm) -- (0:5 cm);

\draw (230:3 cm) -- (230:2.5 cm);
\draw (180:3 cm) -- (180:2.5 cm);
\draw (320:3 cm) -- (320:2.5 cm);
\draw (350:3 cm) -- (350:2.5 cm);

\node[ver] () at (210:5 cm){$\bullet$};
\node[ver] () at (180:3 cm){$\bullet$};
\node[ver] () at (210:5 cm){$\bullet$};
\node[ver] () at (230:3 cm){$\bullet$};
\node[ver] () at (240:5 cm){$\bullet$};
\node[ver] () at (260:3 cm){$\bullet$};
\node[ver] () at (290:3 cm){$\bullet$};
\node[ver] () at (320:3 cm){$\bullet$};
\node[ver] () at (330:5 cm){$\bullet$};
\node[ver] () at (270:5 cm){$\bullet$};
\node[ver] () at (300:5 cm){$\bullet$};
\node[ver] () at (0:5 cm){$\bullet$};
\node[ver] () at (350:3 cm){$\bullet$};
\node[ver] () at (320:3 cm){$\bullet$};

\node[ver] () at (205:6 cm){$w_{k-1}$};
\node[ver] () at (235:5.7 cm){$v_{2k-1}$};
\node[ver] () at (275:5.7 cm){$v_{2k}$};
\node[ver] () at (300:5.7 cm){$v_{1}$};
\node[ver] () at (330:5.7 cm){$v_{2}$};
\node[ver] () at (0:5.7 cm){$w_{1}$};


\end{tikzpicture}
\hspace{.5cm}
\begin{tikzpicture}[scale=0.3]

\draw [thick,dashed] (5,0) arc (0:180:5cm);\draw [thick,dashed] (3,0) arc (0:180:3cm);

\draw (180:5 cm) -- (210:5 cm) -- (240:5 cm) -- (270:5 cm) -- (300:5 cm)-- (330:5 cm)-- (0:5 cm);

\draw (180:3 cm) -- (210:3 cm) -- (240:3 cm) -- (270:3 cm) -- (300:3 cm)-- (330:3 cm)-- (0:3 cm);

\draw (210:5 cm) -- (180:3 cm);
\draw (210:5 cm) -- (210:3 cm);
\draw (240:5 cm) -- (240:3 cm);
\draw (270:5 cm) -- (240:3 cm);
\draw (300:5 cm) -- (270:3 cm);
\draw (300:5 cm) -- (300:3 cm);
\draw (330:5 cm) -- (330:3 cm);
\draw (0:5 cm) -- (330:3 cm);

\draw (210:3 cm) -- (210:2.5 cm);
\draw (180:3 cm) -- (180:2.5 cm);
\draw (270:3 cm) -- (270:2.5 cm);
\draw (300:3 cm) -- (300:2.5 cm);
\draw (180:5 cm) -- (175:4.5 cm);

\node[ver] () at (210:5 cm){$\bullet$};
\node[ver] () at (180:3 cm){$\bullet$};
\node[ver] () at (210:3 cm){$\bullet$};
\node[ver] () at (240:3 cm){$\bullet$};
\node[ver] () at (240:5 cm){$\bullet$};
\node[ver] () at (270:3 cm){$\bullet$};
\node[ver] () at (270:5 cm){$\bullet$};
\node[ver] () at (300:3 cm){$\bullet$};
\node[ver] () at (300:5 cm){$\bullet$};
\node[ver] () at (330:3 cm){$\bullet$};
\node[ver] () at (330:5 cm){$\bullet$};
\node[ver] () at (0:5 cm){$\bullet$};
\node[ver] () at (180:5 cm){$\bullet$};
\node[ver] () at (330:3 cm){$\bullet$};

\node[ver] () at (180:6.3 cm){$v_{2k}$};
\node[ver] () at (205:5.7 cm){$w_{k}$};
\node[ver] () at (235:5.7 cm){$v_{1}$};
\node[ver] () at (275:5.7 cm){$v_{2}$};
\node[ver] () at (300:5.7 cm){$w_{1}$};
\node[ver] () at (330:5.7 cm){$v_{3}$};
\node[ver] () at (0:5.7 cm){$v_{4}$};


\end{tikzpicture}
\hspace{.5cm}
\begin{tikzpicture}[scale=0.3]

\draw [thick,dashed] (5,0) arc (0:180:5cm);\draw [thick,dashed] (3,0) arc (0:180:3cm);

\draw (180:5 cm) -- (210:5 cm) -- (240:5 cm) -- (270:5 cm) -- (300:5 cm)-- (330:5 cm)-- (0:5 cm);

\draw (180:3 cm) -- (180:3 cm) -- (210:3 cm) -- (240:3 cm) -- (270:3 cm)-- (300:3 cm)-- (0:3 cm);

\draw (180:5 cm) -- (180:3 cm);
\draw (210:5 cm) -- (180:3 cm);
\draw (240:5 cm) -- (210:3 cm);
\draw (240:5 cm) -- (240:3 cm);
\draw (270:5 cm) -- (270:3 cm);
\draw (270:5 cm) -- (300:3 cm);
\draw (300:5 cm) -- (0:3 cm);
\draw (330:5 cm) -- (0:3 cm);

\draw (210:3 cm) -- (210:2.5 cm);
\draw (240:3 cm) -- (240:2.5 cm);
\draw (270:3 cm) -- (270:2.5 cm);
\draw (300:3 cm) -- (300:2.5 cm);

\node[ver] () at (180:5 cm){$\bullet$};
\node[ver] () at (180:3 cm){$\bullet$};
\node[ver] () at (210:5 cm){$\bullet$};
\node[ver] () at (180:3 cm){$\bullet$};
\node[ver] () at (240:5 cm){$\bullet$};
\node[ver] () at (210:3 cm){$\bullet$};
\node[ver] () at (270:5 cm){$\bullet$};
\node[ver] () at (270:3 cm){$\bullet$};
\node[ver] () at (300:5 cm){$\bullet$};
\node[ver] () at (300:3 cm){$\bullet$};
\node[ver] () at (330:5 cm){$\bullet$};
\node[ver] () at (240:3 cm){$\bullet$};
\node[ver] () at (0:3 cm){$\bullet$};

\node[ver] () at (180:6.3 cm){$v_{2k-1}$};
\node[ver] () at (205:5.7 cm){$v_{2k}$};
\node[ver] () at (235:5.7 cm){$w_{k}$};
\node[ver] () at (275:5.7 cm){$w_{k+1}$};
\node[ver] () at (300:5.7 cm){$v_{1}$};
\node[ver] () at (330:5.7 cm){$v_{2}$};


\end{tikzpicture}
\vspace{-2mm}
\caption*{{\bf Figure 5: Some more maps on the 2-disc}}

\end{figure}


By the claim, all the vertices on $P_0$, $Q_0$ and $Q_1$ are distinct.

Now $A_{0,0}$, $F_{0,0}$ are faces through $a_{0,1}$. The other two faces through $a_{0,1}$ must be of the form $B_{0,1} := C_4(a_{0,1}, b_{0,2}, a_{0,2}, b_{0,1})$ and $C_{0,1} := C_3(a_{0,1}, c_{1,1}, b_{0,2})$. Then the fourth face through $c_{1,1}$  must be $D_{0,1} := C_4(c_{1,1}, b_{0,2}, c_{0,2}, b_{1,1})$. Similarly, the faces through $b_{1,1}$ are $A_{1,0}$, $D_{0,1}$, $E_{0,1} := C_3(b_{1,1}, c_{0,2}, a_{1,2})$ and  $B_{1,1} := C_4(b_{1,1}, a_{1,2}, b_{1,2}, a_{1,1})$. Continuing this way we get a path $P_1$ and faces $B_{i,1}$, $C_{i,1}$, $D_{i,1}$, $E_{i,1}$, $i \in \mathbb{Z}$, between $Q_1$ and $P_1$ (see Fig. 3).
Where, $B_{i,j}:=$ $C_4(a_{i,2j}, b_{i,2j}, a_{i,2j-1}, b_{i,2j-1})$, $C_{i,j} :=$ $C_3(b_{i,2j}, c_{i+1,2j-1}, a_{i,2j-1})$, $D_{i,j} :=$ $C_4(b_{i,2j}, c_{i,2j}, b_{i+1,2j-1}, c_{i+1,2j-1})$, $E_{i,j} :=$ $C_3(c_{i,2j}, a_{i+1,2j}, b_{i+1,2j-1})$, $i,j\in \mathbb{Z}$.

If we now consider the edge $a_{0,2}b_{0,2}$ on the path $P_1$ we get new path $Q_2 := \cdots   \mbox{-}c_{-1,3} \mbox{-} b_{-1,3}$ $\mbox{-}a_{-1,3}\mbox{-}c_{0,3}\mbox{-}b_{0,3} \mbox{-}a_{0,3}\mbox{-}c_{1,3}\mbox{-}b_{1,3}\mbox{-}a_{1,3}\mbox{-} \cdots$ and faces $A_{i,1}$, $F_{i,1}$, $i \in \mathbb{Z}$, between $P_1$ and $Q_2$.
Where, $A_{i,j} :=$ $C_6(a_{i,2j}, b_{i,2j}, c_{i,2j}, a_{i,2j+1}, b_{i,2j+1}, c_{i,2j+1})$, $F_{i,j} :=$ $C_4(c_{i,2j}$, $a_{i+1,2j}$, $c_{i+1,2j+1}$, $a_{i,2j+1})$, $i,j\in \mathbb{Z}$. Continuing this way we get the map $X$ is as in  Fig. 3.
Then $a_{i,j} \mapsto u_{i,j}$, $b_{i,j} \mapsto v_{i,j}$, $c_{i,j} \mapsto  w_{i,j}$, for all $i, j$, define an isomorphism from $X$ to $E_8$. This proves the lemma.
\end{proof}

\begin{lemma} \label{lemma:E9}
Let $E_9$ be as in Example $\ref{exam:plane}$.
Let $X$ be a semi-equivelar map on the plane. If the type of $X$ is $[3^1, 12^2]$ then $X\cong E_9$.
\end{lemma}

To prove Lemma \ref{lemma:E9}, we need the following.

\begin{lemma} \label{lemma:disc3122}
Let $M$ be a map on the $2$-disk $\mathbb{D}^2$ whose faces are triangles, $12$-gons. For a vertex $u$ of $M$ and $i=3, 12$, let $n_i(u)$ be the number $i$-gons through $u$. Then $M$ can not satisfy both the following: {\rm (i)} $(n_3(u), n_{12}(u)) = (1, 2)$ for each internal vertex $u$, {\rm (ii)} boundary $\partial \mathbb{D}^2$ has $2k+2\ell$ vertices in which $2k$ vertices which are in one triangle and $12$-gon, and $2\ell$ vertices which are in one $12$-gon only, for some $k > 0$ and $k-1 \leq \ell \le k+3$.
\end{lemma}

\begin{proof}
Let $f_0, f_1$ and $f_2$ denote the number of vertices, edges and faces of $M$ respectively. For $i= 3, 12$, let $n_i$ denote the total number of $i$-gons in $M$.
Let there be $n$ internal vertices. Then $f_0=n+2k+2\ell$, $f_2=n_3+n_{12}$ and $f_1=(3n+6k+4\ell)/2$. So, $n$ is even.

Suppose $M$ satisfies (i) and (ii).
Then $n_3=(n+2k)/3$, $n_{12}=(2n+2k+2\ell)/12$. Now $f_2=n_3+n_{12} = (n+2k)/3 + (2n+2k+2\ell)/12 = (6n+10k+2\ell)/12 = k + (3n+t)/6$, $-1 \leq t\leq 3$. Since $n$ is even, this implies that $t = 0$. So, $\ell = k$.
Therefore, $f_0 = n+4k$, $f_2= k + n/2$ and $f_1 = (3n+6k+4k)/2 = n+n/2+5k$. Hence  $f_0-f_1+f_2 = (n+4k) -(n+n/2+5k) + (n/2+k) = 0$. This is not possible since the Euler characteristic of the 2-disk $\mathbb{D}^2$ is 1. This completes the proof.
\end{proof}



\begin{figure}[ht]
\tiny
\tikzstyle{ver}=[]
\tikzstyle{vert}=[circle, draw, fill=black!100, inner sep=0pt, minimum width=4pt]
\tikzstyle{vertex}=[circle, draw, fill=black!00, inner sep=0pt, minimum width=4pt]
\tikzstyle{edge} = [draw,thick,-]
\centering

\begin{tikzpicture}[scale=0.65]





\draw [xshift = -110] ({2*cos(90)},{2*sin(90)}) -- ({2*cos(75)},{2*sin(75)}) -- ({2*cos(45)},{2*sin(45)}) -- ({2*cos(15)},{2*sin(15)}) -- ({2*cos(345)},{2*sin(345)}) -- ({2*cos(315)},{2*sin(315)}) -- ({2*cos(285)},{2*sin(285)}) -- ({2*cos(270)},{2*sin(270)});

\draw ({2*cos(15)},{2*sin(15)}) -- ({2*cos(45)},{2*sin(45)}) -- ({2*cos(75)},{2*sin(75)}) --  ({2*cos(105)},{2*sin(105)}) -- ({2*cos(135)},{2*sin(135)}) -- ({2*cos(165)},{2*sin(165)}) -- ({2*cos(195)},{2*sin(195)}) -- ({2*cos(225)},{2*sin(225)}) -- ({2*cos(255)},{2*sin(255)}) -- ({2*cos(285)},{2*sin(285)}) -- ({2*cos(315)},{2*sin(315)}) -- ({2*cos(345)},{2*sin(345)}) -- ({2*cos(15)},{2*sin(15)});

\draw [xshift = 110] ({2*cos(15)},{2*sin(15)}) -- ({2*cos(45)},{2*sin(45)}) -- ({2*cos(75)},{2*sin(75)}) --  ({2*cos(105)},{2*sin(105)}) -- ({2*cos(135)},{2*sin(135)}) -- ({2*cos(165)},{2*sin(165)}) -- ({2*cos(195)},{2*sin(195)}) -- ({2*cos(225)},{2*sin(225)}) -- ({2*cos(255)},{2*sin(255)}) -- ({2*cos(285)},{2*sin(285)}) -- ({2*cos(315)},{2*sin(315)}) -- ({2*cos(345)},{2*sin(345)}) -- ({2*cos(15)},{2*sin(15)});

\draw [xshift = 220] ({2*cos(15)},{2*sin(15)}) -- ({2*cos(45)},{2*sin(45)}) -- ({2*cos(75)},{2*sin(75)}) --  ({2*cos(105)},{2*sin(105)}) -- ({2*cos(135)},{2*sin(135)}) -- ({2*cos(165)},{2*sin(165)}) -- ({2*cos(195)},{2*sin(195)}) -- ({2*cos(225)},{2*sin(225)}) -- ({2*cos(255)},{2*sin(255)}) -- ({2*cos(285)},{2*sin(285)}) -- ({2*cos(315)},{2*sin(315)}) -- ({2*cos(345)},{2*sin(345)}) -- ({2*cos(15)},{2*sin(15)});


\draw [xshift = -55, yshift = 95] ({2*cos(15)},{2*sin(15)}) -- ({2*cos(45)},{2*sin(45)}) -- ({2*cos(75)},{2*sin(75)}) --  ({2*cos(105)},{2*sin(105)}) -- ({2*cos(135)},{2*sin(135)}) -- ({2*cos(165)},{2*sin(165)}) -- ({2*cos(195)},{2*sin(195)}) -- ({2*cos(225)},{2*sin(225)}) -- ({2*cos(255)},{2*sin(255)}) -- ({2*cos(285)},{2*sin(285)}) -- ({2*cos(315)},{2*sin(315)}) -- ({2*cos(345)},{2*sin(345)}) -- ({2*cos(15)},{2*sin(15)});

\draw [xshift = 55, yshift = 95] ({2*cos(15)},{2*sin(15)}) -- ({2*cos(45)},{2*sin(45)}) -- ({2*cos(75)},{2*sin(75)}) --  ({2*cos(105)},{2*sin(105)}) -- ({2*cos(135)},{2*sin(135)}) -- ({2*cos(165)},{2*sin(165)}) -- ({2*cos(195)},{2*sin(195)}) -- ({2*cos(225)},{2*sin(225)}) -- ({2*cos(255)},{2*sin(255)}) -- ({2*cos(285)},{2*sin(285)}) -- ({2*cos(315)},{2*sin(315)}) -- ({2*cos(345)},{2*sin(345)}) -- ({2*cos(15)},{2*sin(15)});

\draw [xshift = 165, yshift = 95] ({2*cos(15)},{2*sin(15)}) -- ({2*cos(45)},{2*sin(45)}) -- ({2*cos(75)},{2*sin(75)}) --  ({2*cos(105)},{2*sin(105)}) -- ({2*cos(135)},{2*sin(135)}) -- ({2*cos(165)},{2*sin(165)}) -- ({2*cos(195)},{2*sin(195)}) -- ({2*cos(225)},{2*sin(225)}) -- ({2*cos(255)},{2*sin(255)}) -- ({2*cos(285)},{2*sin(285)}) -- ({2*cos(315)},{2*sin(315)}) -- ({2*cos(345)},{2*sin(345)}) -- ({2*cos(15)},{2*sin(15)});

\draw [xshift = 275, yshift = 95] ({2*cos(90)},{2*sin(90)}) -- ({2*cos(105)},{2*sin(105)}) -- ({2*cos(135)},{2*sin(135)}) -- ({2*cos(165)},{2*sin(165)}) -- ({2*cos(195)},{2*sin(195)}) -- ({2*cos(225)},{2*sin(225)}) -- ({2*cos(255)},{2*sin(255)}) -- ({2*cos(270)},{2*sin(270)});


\draw [xshift = -110, yshift = 190]({2*cos(315)},{2*sin(315)}) -- ({2*cos(285)},{2*sin(285)}) -- ({2*cos(270)},{2*sin(270)});

\draw [yshift = 190] ({2*cos(255)},{2*sin(255)}) -- ({2*cos(285)},{2*sin(285)});

\draw [xshift = 110, yshift = 190] ({2*cos(225)},{2*sin(225)}) -- ({2*cos(255)},{2*sin(255)}) -- ({2*cos(285)},{2*sin(285)});

\draw [xshift = 220, yshift = 190] ({2*cos(225)},{2*sin(225)}) -- ({2*cos(255)},{2*sin(255)}) -- ({2*cos(285)},{2*sin(285)});


\draw [xshift = -55, yshift = -95] ({2*cos(75)},{2*sin(75)}) --  ({2*cos(105)},{2*sin(105)}) -- ({2*cos(135)},{2*sin(135)}) ;

\draw [xshift = 55, yshift = -95]  ({2*cos(75)},{2*sin(75)}) --  ({2*cos(105)},{2*sin(105)}) -- ({2*cos(135)},{2*sin(135)});

\draw [xshift = 165, yshift = -95]  ({2*cos(75)},{2*sin(75)}) --  ({2*cos(105)},{2*sin(105)}) -- ({2*cos(135)},{2*sin(135)});

\draw [xshift = 275, yshift = -95] ({2*cos(90)},{2*sin(90)}) -- ({2*cos(105)},{2*sin(105)}) -- ({2*cos(135)},{2*sin(135)});

\node[ver] () at (1.1-4.2,-1.2){$c_{-1,-1}$};
\node[ver] () at (.5-4.2,-1.7){$c_{-2,-1}$};
\node[ver] () at (1.1-4.2,1.2){$b_{-1,0}$};
\node[ver] () at (.5-4.2,1.7){$b_{-2,0}$};

\node[ver] () at (-2.8,-.5){$a_{-1,-1}$};
\node[ver] () at (-1.1,-1.1){$b_{0,-1}$};
\node[ver] () at (-.5,-1.6){$b_{1,-1}$};
\node[ver] () at (1.3,-.5){$a_{1,-1}$};
\node[ver] () at (1.1,-1.1){$c_{1,-1}$};
\node[ver] () at (.5,-1.6){$c_{0,-1}$};
\node[ver] () at (1.4,.5){$a_{0,0}$};
\node[ver] () at (1.1,1.1){$b_{1,0}$};
\node[ver] () at (.6,1.6){$b_{0,0}$};
\node[ver] () at (-2.6,.5){$a_{-2,0}$};
\node[ver] () at (-1,1.1){$c_{-2,0}$};
\node[ver] () at (-.3,1.6){$c_{-1,0}$};

\node[ver] () at (-1.1+3.9,-1.1){$b_{2,-1}$};
\node[ver] () at (-.5+3.8,-1.6){$b_{3,-1}$};
\node[ver] () at (1.4+3.8,-.5){$a_{3,-1}$};
\node[ver] () at (1.1+3.8,-1.1){$c_{3,-1}$};
\node[ver] () at (.5+3.8,-1.6){$c_{2,-1}$};
\node[ver] () at (1.4+3.8,.5){$a_{2,0}$};
\node[ver] () at (1.1+3.8,1.1){$b_{3,0}$};
\node[ver] () at (.6+3.9,1.6){$b_{2,0}$};
\node[ver] () at (-.7+3.8,1.3){$c_{0,0}$};
\node[ver] () at (-.4+3.9,1.6){$c_{1,0}$};

\node[ver] () at (-1+7.7,-1.1){$b_{4,-1}$};
\node[ver] () at (-.5+7.6,-1.6){$b_{5,-1}$};
\node[ver] () at (1.4+7.6,-.5){$a_{5,-1}$};
\node[ver] () at (1.1+7.6,-1.1){$c_{5,-1}$};
\node[ver] () at (.5+7.6,-1.6){$c_{4,-1}$};
\node[ver] () at (1.5+7.6,.5){$a_{4,0}$};
\node[ver] () at (1.1+7.6,1.1){$b_{5,0}$};
\node[ver] () at (.5+7.8,1.6){$b_{4,0}$};
\node[ver] () at (-.7+7.6,1.3){$c_{2,0}$};
\node[ver] () at (-.4+7.8,1.6){$c_{3,0}$};

\node[ver] () at (.9-4.1,-1.2 + 4){$a_{-3,0}$};

\node[ver] () at (0.7,-1.2 + 4){$a_{-1,0}$};

\node[ver] () at (4.4,-1.2 + 4){$a_{1,0}$};

\node[ver] () at (8.3,-1.2 + 4){$a_{3,0}$};

\node[ver] () at (1.7-4.2,5.5){$c_{-3,1}$};
\node[ver] () at (-2.7,-1.6+6.2){$c_{-4,1}$};

\node[ver] () at (-1.4,5.5){$b_{-2,1}$};
\node[ver] () at (-1,-1.6+6.2){$b_{-1,1}$};
\node[ver] () at (1.4,-1.1+6.6){$c_{-1,1}$};
\node[ver] () at (1.2,-1.6+6.2){$c_{-2,1}$};

\node[ver] () at (-1.1+3.6,5.5){$b_{0,1}$};
\node[ver] () at (-.5+3.5,-1.6+6.2){$b_{1,1}$};
\node[ver] () at (1.5+3.8,5.5){$c_{1,1}$};
\node[ver] () at (4.8,-1.6+6.2){$c_{0,1}$};

\node[ver] () at (-1+7.2,5.5){$b_{2,1}$};
\node[ver] () at (-.5+7.4,-1.6+6.2){$b_{3,1}$};
\node[ver] () at (1.4+7.9,5.5){$c_{3,1}$};
\node[ver] () at (1.1+7.6,-1.6+6.2){$c_{2,1}$};

\node[ver] () at (1.1-4.2,-1.2 - 1.7+6.7){$a_{-4,1}$};
\node[ver] () at (1.1-.4,-1.2 - 1.7+6.7){$a_{-2,1}$};
\node[ver] () at (1+3.4,-1.2 - 1.7+6.7){$a_{0,1}$};
\node[ver] () at (1.2+7,-1.2 - 1.7+6.7){$a_{2,1}$};
\node[ver] () at (0,0){$A_{0,-1}$};
\node[ver] () at (4,0){$A_{1,-1}$};
\node[ver] () at (7.8,0){$A_{2,-1}$};

\node[ver] () at (-1.5,3.5){$A_{-1,0}$};
\node[ver] () at (2,3.5){$A_{0,0}$};
\node[ver] () at (6,3.5){$A_{1,0}$};
\node[ver] () at (9.5,3.5){$A_{2,0}$};

\node[ver] () at (0,2.1){$D_{0,0}$};
\node[ver] () at (3.9,2.1){$D_{1,0}$};
\node[ver] () at (7.8,2.1){$D_{2,0}$};

\node[ver] () at (2,1.1){$F_{0,0}$};
\node[ver] () at (5.8,1.1){$F_{1,0}$};

\node[ver] () at (4,4.4){$F_{0,1}$};
\node[ver] () at (7.8,4.4){$F_{1,1}$};

\node[ver] () at (10.3,1.3){$P_{0}$};
\node[ver] () at (10.3,-1.3){$P_{-1}$};
\node[ver] () at (10.3,5.2){$P_{1}$};

\node[ver] () at (2, -3.3){\normalsize \bf Figure 6:  Part of {\boldmath $X$} of Lemma \ref{lemma:E9}};

\end{tikzpicture}

\end{figure}


\begin{proof}[Proof of Lemma \ref{lemma:E9}]
Let $b_{0,0}b_{1,0}$ be an edge of $X$ in two 12-gons. Assume that these 12-gons are
 $A_{0,0} := C_{12}(b_{0,0}, b_{1,0}, c_{0,0}, c_{1,0}, a_{1,0}, a_{0,1}, b_{1,1}, b_{0,1}, c_{-1,1}, c_{-2,1}, a_{-2,1}, a_{-1,0})$ and  $A_{0,-1} := C_{12}(b_{0,-1}, b_{1,-1}, c_{0,-1}, c_{1,-1}, a_{1,-1}, a_{0,0}, b_{1,0}, b_{0,0}, c_{-1,0}, c_{-2,0}, a_{-2,0}, a_{-1,-1})$ (see Fig. 6). This implies that $D_{0,0}:= C_{3}(b_{0,0}, a_{-1,0}, c_{-1,0})$ and $F_{0,0} := C_{3}(b_{1,0}, a_{0,0}, c_{0,0})$ are faces (through $b_{0,0}$ and $b_{1,0}$ respectively).  Clearly, the second face containing the edge $a_{0,0}c_{0,0}$ must be of the form $A_{1,-1} := C_{12}( b_{2,-1}, b_{3,-1}, c_{2,-1}, c_{3,-1}, a_{3,-1}, a_{2,0}, b_{3,0}, b_{2,0}, c_{1,0}, c_{0,0}, a_{0,0}, a_{1,-1})$.
Then $D_{1,0} := C_{3}(c_{1,0}, b_{2,0}, a_{1,0})$ and $D_{1,-1} := C_{3}(c_{1,-1}, b_{2,-1}, a_{1,-1})$ are faces. Again, the second face through $b_{2,0}b_{3,0}$ must be of the form $A_{1,0} := C_{12}(b_{2,0}, b_{3,0}, c_{2,0}, c_{3,0}$, $a_{3,0}$, $a_{2,1}$, $b_{3,1}, b_{2,1}, c_{1,1}, c_{0,1}, a_{0,1}, a_{1,0})$. Then $F_{1,0} := C_3(b_{3,0}, a_{2,0}, c_{2,0})$ and $F_{0,1} := C_3(b_{1,1}, a_{0,1}, c_{0,1})$ are faces.
Continuing this way get the paths $P_j := \cdots\mbox{-}b_{-2,j}\mbox{-}b_{-1,j}\mbox{-} c_{-2,j}\mbox{-}c_{-1,j}\mbox{-}b_{0,j}\mbox{-}b_{1,j}\mbox{-}c_{0,j}$ $\mbox{-}c_{1,j}\mbox{-} b_{2,j}\mbox{-}b_{3,j}\mbox{-}\cdots$ for $-1\leq j\leq 1$.

\medskip

\noindent {\bf Claim.} $P_0$, $P_{-1}$, $P_1$ are infinite paths.

\smallskip

If not then we get a cycle (as a subgraph of $P_0$). This gives (by the similar arguments as in the proof of Lemma \ref{lemma:E8}) a map $M$ on the disc $\mathbb{D}^2$ which satisfies properties (i) and (ii) of Lemma \ref{lemma:disc3122} for some $k, \ell$. But, this is not possible by Lemma \ref{lemma:disc3122}. Thus, $P_0$ is an infinite path. Similarly, $P_{-1}$ and $P_1$ are infinite paths.
This proves the claim.

\smallskip

By the claim, all the vertices on $P_0$, $P_{-1}$ and $P_1$ are distinct. Above arguments also give faces (between the paths $P_{-1}$ and $P_1$) $A_{i,j}, D_{i,j}$,  $F_{i,j+1}$, for $j= -1, 0$, $i\in \mathbb{Z}$, where $A_{i,j} :=$ $C_{12}(b_{2i,j}, b_{2i+1,j}, c_{2i,j}, c_{2i+1,j}, a_{2i+1,j}, a_{2i,j+1}, b_{2i+1,j+1}, b_{2i,j+1}$, $c_{2i-1,j+1}$, $c_{2i-2,j+1}$, $a_{2i-2,j+1}$, $a_{2i-1,j})$, $D_{i,j}:= C_3(c_{2i-1,j}, b_{2i,j}, a_{2i-1,j})$, $F_{i,k}:= C_3(b_{2i+1,k}, a_{2i,k}, c_{2i,k})$.

Similarly, if we start with the edge $b_{0,1}b_{1,1}$ in place of $b_{0,0}b_{1,0}$, we get
faces $A_{i,j}, D_{i,j}$ and $F_{i,j+1}$, for $j= 0, 1$, $i\in \mathbb{Z}$ between the paths $P_{0}$ and $P_2$ (with same common faces between $P_{0}$ and $P_1$).
Continuing this way we get faces $A_{i,j}, D_{i,j}, F_{i,j}$ of $X$, for $i, j \in \mathbb{Z}$. Then the mapping $\varphi \colon X \mapsto E_9$, given by $\varphi(a_{i,j}) = u_{i,j}$, $\varphi(b_{i,j}) = v_{i,j}$, $\varphi(c_{i,j}) = w_{i,j}$, is an isomorphism. This completes the proof.
\end{proof}

\begin{lemma} \label{lemma:E10}
Let $E_{10}$ be as in Example $\ref{exam:plane}$.
Let $X$ be a semi-equivelar map on the plane. If the type of $X$ is $[4^1, 6^1, 12^1]$ then $X\cong E_{10}$.
\end{lemma}

Again, to prove Lemma \ref{lemma:E10} we need the following two technical lemmas.

\begin{lemma} \label{lemma:disc4612-1}
Let $M$ be a map on the $2$-disk $\mathbb{D}^2$ whose faces are quadrangles, hexagons, $12$-gons. For a vertex $u$ of $M$ and $i=4, 6, 12$, let $n_i(u)$ be the number $i$-gons through $u$. Then $M$ can not satisfy both the following: {\rm (i)} $(n_4(u), n_6(u), n_{12}(u)) = (1, 1, 1)$ for each internal vertex $u$, {\rm (ii)} boundary $\partial \mathbb{D}^2$ contains $2k+2\ell$ vertices, among these  $2k$ vertices are in one quadrangle and one $12$-gon, and $2\ell$ vertices are in one $12$-gon only, for some $k > 0$ and $2k-2 \le \ell \le 2k+2$.
\end{lemma}

\begin{proof}
Let $f_0, f_1$ and $f_2$ denote the number of vertices, edges and faces of $X$ respectively. For $i= 4, 6, 12$, let $n_i$ denote the total number of $i$-gons in $X$.
Let there be $n$ internal vertices. Thus,  $f_0=n+2k+2\ell$ and $f_2=n_4+n_6+n_{12}$.

Suppose $X$ satisfies (i) and (ii). Then $n_6=n/6$, $n_4=(n+2k)/4 =(6n_6+2k)/4 = (3n_6+k)/2$ and $n_{12}=(n+2k+2\ell)/12 = (6n_6+2k+2\ell)/12 = (3n_6+k+\ell)/6=(2n_4+\ell)/6$. So, $\ell$ is even.
Now, $f_2 = n_4 + n_6 + n_{12} = (3n_6+k)/2 + n_6 +(3n_6+k+\ell)/6 = 3n_6 + (4k+\ell)/6$. So, $(4k+\ell)/6\in\mathbb{Z}$, where $\ell$ is even and $2k-2\leq \ell \leq 2k+2$.  This implies that $\ell = 2k$.
Therefore $f_0 = 6n_6+6k$, $f_2= 3n_6 + k$ and $f_1 = (3n+6k+4\ell)/2= (3n + 6k +8k)/2 = 9n_6 + 7k$. Hence  $f_0-f_1+f_2 = (6n_6+6k) -(9n_6 + 7k) + (3n_6+k) = 0$. This is not possible since the Euler characteristic of $\mathbb{D}^2$ is 1. This completes the proof.
\end{proof}

\begin{lemma} \label{lemma:disc4612-2}
Let $M$ be a map on the $2$-disk $\mathbb{D}^2$ whose faces are quadrangles, hexagons, $12$-gons. For a vertex $u$ of $M$ and $i=4, 6, 12$, let $n_i(u)$ be the number $i$-gons through $u$. Then $M$ can not satisfy both the following: {\rm (i)} $(n_4(u), n_6(u), n_{12}(u)) = (1, 1, 1)$ for each internal vertex $u$, {\rm (ii)} boundary $\partial \mathbb{D}^2$ contains $4k+2\ell$ vertices, among these  $4k$ vertices are in one quadrangle and one hexagon, and $2\ell$ vertices are in one hexagon only, for some $k > 0$ and $k-1 \le \ell \le k$.
\end{lemma}

\begin{proof}
Let $f_0, f_1$ and $f_2$ denote the number of vertices, edges and faces of $X$ respectively. For $i= 4, 6, 12$, let $n_i$ denote the total number of $i$-gons in $X$. Let there be $n$ internal vertices. Then $f_0=n+4k+2\ell$ and $f_2=n_4+n_6+n_{12}$.

Suppose $X$ satisfies (i) and (ii). Then $n_{12} = n/12, n_4 = (n+4k)/4 = (12n_{12}+4k)/4 = 3n_{12}+k$, $n_6 = (n+4k+2\ell)/6 = (12n_{12}+4k+2\ell)/6 = 2n_{12}+(2k+\ell)/3$. This implies that $\ell = k$ since $k-1 \le \ell \le k$.
Therefore, $f_0 = n+6k = 12n_{12} + 6k$, $f_2= n_4 + n_6 +n_{12} = (3n_{12} + k) + (2n_{12}+k) + n_{12} = 6n_{12} + 2k$ and $f_1= (3n+12k+4k)/2 = 18n_{12}+6k+2k$. So, $f_0-f_1+f_2 = (12n_{12}+6k) -(18n_{12}+8k) + (6n_{12}+2k) = 0$. This is not possible since the Euler characteristic of the 2-disk $\mathbb{D}^2$ is 1. This completes the proof.
\end{proof}

\begin{figure}[ht]
\tiny
\tikzstyle{ver}=[]
\tikzstyle{vert}=[circle, draw, fill=black!100, inner sep=0pt, minimum width=4pt]
\tikzstyle{vertex}=[circle, draw, fill=black!00, inner sep=0pt, minimum width=4pt]
\tikzstyle{edge} = [draw,thick,-]
\centering

\begin{tikzpicture}[scale=0.57]
\draw [xshift = -80, yshift = 8] ({1*cos(25)},{1*sin(14)}) -- ({1*cos(100)},{1*sin(14)});

\draw [xshift = -140] ({2*cos(90)},{2*sin(90)}) -- ({2*cos(75)},{2*sin(75)}) -- ({2*cos(45)},{2*sin(45)}) -- ({2*cos(15)},{2*sin(15)});

\draw [xshift = 60, yshift = 8] ({1*cos(25)},{1*sin(14)}) -- ({1*cos(100)},{1*sin(14)});

\draw ({2*cos(15)},{2*sin(15)}) -- ({2*cos(45)},{2*sin(45)}) -- ({2*cos(75)},{2*sin(75)}) --  ({2*cos(105)},{2*sin(105)}) -- ({2*cos(135)},{2*sin(135)}) -- ({2*cos(165)},{2*sin(165)});

\draw [xshift = -78, yshift = -22] (-1.7,2.7) -- (-1.2,3.6);
\draw [xshift = -90, yshift = 10] (-.35,1.05) -- (0.1,1.95);

\draw [xshift = 62, yshift = -22] (-1.7,2.7) -- (-1.2,3.6);
\draw [xshift = 50, yshift = 10] (-.35,1.05) -- (0.1,1.95);

\draw [xshift = 202, yshift = -22] (-1.7,2.7) -- (-1.2,3.6);
\draw [xshift = 190, yshift = 10] (-.35,1.05) -- (0.1,1.95);

\draw [xshift = 342, yshift = -22] (-1.7,2.7) -- (-1.2,3.6);
\draw [xshift = 330, yshift = 10] (-.35,1.05) -- (0.1,1.95);

\draw [xshift = 200, yshift = 8] ({1*cos(25)},{1*sin(14)}) -- ({1*cos(100)},{1*sin(14)});

\draw [xshift = 140] ({2*cos(15)},{2*sin(15)}) -- ({2*cos(45)},{2*sin(45)}) -- ({2*cos(75)},{2*sin(75)}) --  ({2*cos(105)},{2*sin(105)}) -- ({2*cos(135)},{2*sin(135)}) -- ({2*cos(165)},{2*sin(165)});

\draw [xshift = 280] ({2*cos(15)},{2*sin(15)}) -- ({2*cos(45)},{2*sin(45)}) -- ({2*cos(75)},{2*sin(75)}) --  ({2*cos(105)},{2*sin(105)}) -- ({2*cos(135)},{2*sin(135)}) -- ({2*cos(165)},{2*sin(165)});

\draw [xshift = 0, yshift = 50] (-.5,.15) -- (-1.13,1.05);
\draw [xshift = -26, yshift = 35] (-.5,.15) -- (-1.13,1.05);

\draw [xshift = 140, yshift = 50] (-.5,.15) -- (-1.13,1.05);
\draw [xshift = 114, yshift = 35] (-.5,.15) -- (-1.13,1.05);

\draw [xshift = 280, yshift = 50] (-.5,.15) -- (-1.13,1.05);
\draw [xshift = 254, yshift = 35] (-.5,.15) -- (-1.13,1.05);


\draw [xshift = -70, yshift = 170] (-.6,.15) -- (-1.13,1.05);
\draw [xshift = -96, yshift = 155] (-.6,.15) -- (-1.13,1.05);

\draw [xshift = 70, yshift = 170] (-.6,.15) -- (-1.13,1.05);
\draw [xshift = 44, yshift = 155] (-.6,.15) -- (-1.13,1.05);

\draw [xshift = 210, yshift = 170] (-.6,.15) -- (-1.13,1.05);
\draw [xshift = 184, yshift = 155] (-.6,.15) -- (-1.13,1.05);

\draw [xshift = 350, yshift = 170] (-.6,.15) -- (-1.13,1.05);
\draw [xshift = 324, yshift = 155] (-.6,.15) -- (-1.13,1.05);

\draw [xshift = -72, yshift = 120] ({2*cos(15)},{2*sin(15)}) -- ({2*cos(45)},{2*sin(45)}) -- ({2*cos(75)},{2*sin(75)}) --  ({2*cos(105)},{2*sin(105)}) -- ({2*cos(135)},{2*sin(135)}) -- ({2*cos(165)},{2*sin(165)}) -- ({2*cos(195)},{2*sin(195)}) -- ({2*cos(225)},{2*sin(225)}) -- ({2*cos(255)},{2*sin(255)}) -- ({2*cos(285)},{2*sin(285)}) -- ({2*cos(315)},{2*sin(315)}) -- ({2*cos(345)},{2*sin(345)}) -- ({2*cos(15)},{2*sin(15)});

\draw [xshift = -12, yshift = 128] ({1*cos(25)},{1*sin(14)}) -- ({1*cos(100)},{1*sin(14)});
\draw [xshift = -12, yshift = 98] ({1*cos(25)},{1*sin(14)}) -- ({1*cos(100)},{1*sin(14)});

\draw [xshift = 68, yshift = 120] ({2*cos(15)},{2*sin(15)}) -- ({2*cos(45)},{2*sin(45)}) -- ({2*cos(75)},{2*sin(75)}) --  ({2*cos(105)},{2*sin(105)}) -- ({2*cos(135)},{2*sin(135)}) -- ({2*cos(165)},{2*sin(165)}) -- ({2*cos(195)},{2*sin(195)}) -- ({2*cos(225)},{2*sin(225)}) -- ({2*cos(255)},{2*sin(255)}) -- ({2*cos(285)},{2*sin(285)}) -- ({2*cos(315)},{2*sin(315)}) -- ({2*cos(345)},{2*sin(345)}) -- ({2*cos(15)},{2*sin(15)});

\draw [xshift = 128, yshift = 128] ({1*cos(25)},{1*sin(14)}) -- ({1*cos(100)},{1*sin(14)});
\draw [xshift = 128, yshift = 98] ({1*cos(25)},{1*sin(14)}) -- ({1*cos(100)},{1*sin(14)});

\draw [xshift = 208, yshift = 120] ({2*cos(15)},{2*sin(15)}) -- ({2*cos(45)},{2*sin(45)}) -- ({2*cos(75)},{2*sin(75)}) --  ({2*cos(105)},{2*sin(105)}) -- ({2*cos(135)},{2*sin(135)}) -- ({2*cos(165)},{2*sin(165)}) -- ({2*cos(195)},{2*sin(195)}) -- ({2*cos(225)},{2*sin(225)}) -- ({2*cos(255)},{2*sin(255)}) -- ({2*cos(285)},{2*sin(285)}) -- ({2*cos(315)},{2*sin(315)}) -- ({2*cos(345)},{2*sin(345)}) -- ({2*cos(15)},{2*sin(15)});

\draw [xshift = 268, yshift = 128] ({1*cos(25)},{1*sin(14)}) -- ({1*cos(100)},{1*sin(14)});
\draw [xshift = 268, yshift = 98] ({1*cos(25)},{1*sin(14)}) -- ({1*cos(100)},{1*sin(14)});

\draw [xshift = 348, yshift = 120] ({2*cos(90)},{2*sin(90)}) -- ({2*cos(105)},{2*sin(105)}) -- ({2*cos(135)},{2*sin(135)}) -- ({2*cos(165)},{2*sin(165)}) -- ({2*cos(195)},{2*sin(195)}) -- ({2*cos(225)},{2*sin(225)}) -- ({2*cos(255)},{2*sin(255)}) -- ({2*cos(270)},{2*sin(270)});

\draw [xshift = 272, yshift = 98] (-1.7,2.7) -- (-1.25,3.58);
\draw [xshift = 259, yshift = 130] (-.35,1.05) -- (0.1,1.95);

\draw [xshift = 132, yshift = 98] (-1.7,2.7) -- (-1.25,3.58);
\draw [xshift = 119, yshift = 130] (-.35,1.05) -- (0.1,1.95);

\draw [xshift = -8, yshift = 98] (-1.7,2.7) -- (-1.25,3.58);
\draw [xshift = -21, yshift = 130] (-.35,1.05) -- (0.1,1.95);


\draw [xshift = -83, yshift = 248] ({1*cos(25)},{1*sin(14)}) -- ({1*cos(100)},{1*sin(14)});
\draw [xshift = -83, yshift = 218] ({1*cos(25)},{1*sin(14)}) -- ({1*cos(100)},{1*sin(14)});

\draw [xshift = -143, yshift = 240] ({2*cos(90)},{2*sin(90)}) -- ({2*cos(75)},{2*sin(75)}) -- ({2*cos(45)},{2*sin(45)}) -- ({2*cos(15)},{2*sin(15)}) -- ({2*cos(345)},{2*sin(345)}) -- ({2*cos(315)},{2*sin(315)}) -- ({2*cos(285)},{2*sin(285)}) -- ({2*cos(270)},{2*sin(270)});

\draw [xshift = 57, yshift = 248] ({1*cos(25)},{1*sin(14)}) -- ({1*cos(100)},{1*sin(14)});
\draw [xshift = 57, yshift = 218] ({1*cos(25)},{1*sin(14)}) -- ({1*cos(100)},{1*sin(14)});

\draw [xshift = -3, yshift = 240] ({2*cos(15)},{2*sin(15)}) -- ({2*cos(45)},{2*sin(45)}) -- ({2*cos(75)},{2*sin(75)}) --  ({2*cos(105)},{2*sin(105)}) -- ({2*cos(135)},{2*sin(135)}) -- ({2*cos(165)},{2*sin(165)}) -- ({2*cos(195)},{2*sin(195)}) -- ({2*cos(225)},{2*sin(225)}) -- ({2*cos(255)},{2*sin(255)}) -- ({2*cos(285)},{2*sin(285)}) -- ({2*cos(315)},{2*sin(315)}) -- ({2*cos(345)},{2*sin(345)}) -- ({2*cos(15)},{2*sin(15)});

\draw [xshift = 197, yshift = 248] ({1*cos(25)},{1*sin(14)}) -- ({1*cos(100)},{1*sin(14)});
\draw [xshift = 197, yshift = 218] ({1*cos(25)},{1*sin(14)}) -- ({1*cos(100)},{1*sin(14)});

\draw [xshift = 317, yshift = 248] ({1*cos(25)},{1*sin(14)}) -- ({1*cos(60)},{1*sin(14)});
\draw [xshift = 317, yshift = 218] ({1*cos(25)},{1*sin(14)}) -- ({1*cos(60)},{1*sin(14)});

\draw [xshift = 320, yshift = 7] ({1*cos(25)},{1*sin(14)}) -- ({1*cos(60)},{1*sin(14)});

\draw [xshift = -153, yshift = 128] ({1*cos(25)},{1*sin(14)}) -- ({1*cos(60)},{1*sin(14)});
\draw [xshift = -153, yshift = 98] ({1*cos(25)},{1*sin(14)}) -- ({1*cos(60)},{1*sin(14)});

\draw [xshift = 137, yshift = 240] ({2*cos(15)},{2*sin(15)}) -- ({2*cos(45)},{2*sin(45)}) -- ({2*cos(75)},{2*sin(75)}) --  ({2*cos(105)},{2*sin(105)}) -- ({2*cos(135)},{2*sin(135)}) -- ({2*cos(165)},{2*sin(165)}) -- ({2*cos(195)},{2*sin(195)}) -- ({2*cos(225)},{2*sin(225)}) -- ({2*cos(255)},{2*sin(255)}) -- ({2*cos(285)},{2*sin(285)}) -- ({2*cos(315)},{2*sin(315)}) -- ({2*cos(345)},{2*sin(345)}) -- ({2*cos(15)},{2*sin(15)});

\draw [xshift = 277, yshift = 240] ({2*cos(15)},{2*sin(15)}) -- ({2*cos(45)},{2*sin(45)}) -- ({2*cos(75)},{2*sin(75)}) --  ({2*cos(105)},{2*sin(105)}) -- ({2*cos(135)},{2*sin(135)}) -- ({2*cos(165)},{2*sin(165)}) -- ({2*cos(195)},{2*sin(195)}) -- ({2*cos(225)},{2*sin(225)}) -- ({2*cos(255)},{2*sin(255)}) -- ({2*cos(285)},{2*sin(285)}) -- ({2*cos(315)},{2*sin(315)}) -- ({2*cos(345)},{2*sin(345)}) -- ({2*cos(15)},{2*sin(15)});

\node[ver] () at (1.2-5,.5){$a_{-1,-1}$};
\node[ver] () at (1-5,1.3){$b_{-1,-1}$};
\node[ver] () at (.1-5,1.7){$c_{-1,-1}$};

\node[ver] () at (1.2,.5){$a_{0,-1}$};
\node[ver] () at (1,1.1){$b_{0,-1}$};
\node[ver] () at (.5,1.5){$c_{0,-1}$};
\node[ver] () at (-1.2,.5){$d_{0,-1}$};
\node[ver] () at (-1,1.1){$e_{0,-1}$};
\node[ver] () at (-.5,1.5){$f_{0,-1}$};

\node[ver] () at (1.2+5,.5){$a_{1,-1}$};
\node[ver] () at (1+5,1.1){$b_{1,-1}$};
\node[ver] () at (.5+5,1.5){$c_{1,-1}$};
\node[ver] () at (-1.2+5,.5){$d_{1,-1}$};
\node[ver] () at (-1+5,1.1){$e_{1,-1}$};
\node[ver] () at (-.5+5,1.5){$f_{1,-1}$};

\node[ver] () at (1.2+10,.5){$a_{2,-1}$};
\node[ver] () at (1+10,1.1){$b_{2,-1}$};
\node[ver] () at (.5+10,1.5){$c_{2,-1}$};
\node[ver] () at (-1.2+10,.5){$d_{2,-1}$};
\node[ver] () at (-1+10,1.1){$e_{2,-1}$};
\node[ver] () at (-.5+10,1.5){$f_{2,-1}$};

\node[ver] () at (-1.2-2.5,-.5+4){$a_{-1,0}$};
\node[ver] () at (-1-2.5,-1.1+4){$b_{-1,0}$};
\node[ver] () at (-.5-2.5,-1.4+4){$c_{-1,0}$};
\node[ver] () at (1.2-2.5,-.5+4){$d_{-1,0}$};
\node[ver] () at (1-2.5,-1.1+4){$e_{-1,0}$};
\node[ver] () at (.5-2.5,-1.4+4){$f_{-1,0}$};
\node[ver] () at (1.2-2.5,.5+4){$a_{-1,1}$};
\node[ver] () at (1-2.5,1.1+4){$b_{-1,1}$};
\node[ver] () at (.6-2.5,1.6+4){$c_{-1,1}$};
\node[ver] () at (-1.2-2.5,.5+4){$d_{-1,1}$};
\node[ver] () at (-1-2.5,1.1+4){$e_{-1,1}$};
\node[ver] () at (-.6-2.5,1.6+4){$f_{-1,1}$};

\node[ver] () at (-1.4+2.5,-.5+4.2){$a_{0,0}$};
\node[ver] () at (-1.2+2.5,-1.1+4.2){$b_{0,0}$};
\node[ver] () at (-.6+2.5,-1.6+4.2){$c_{0,0}$};
\node[ver] () at (1.2+2.5,-.5+4.2){$d_{0,0}$};
\node[ver] () at (1.1+2.5,-1.1+4.2){$e_{0,0}$};
\node[ver] () at (.5+2.5,-1.6+4.2){$f_{0,0}$};
\node[ver] () at (1.4+2.5,.5+4.2){$a_{0,1}$};
\node[ver] () at (1.1+2.5,1.1+4.2){$b_{0,1}$};
\node[ver] () at (.5+2.5,1.6+4.2){$c_{0,1}$};
\node[ver] () at (-1.4+2.5,.5+4.2){$d_{0,1}$};
\node[ver] () at (-1.2+2.5,1.1+4.2){$e_{0,1}$};
\node[ver] () at (-.6+2.5,1.6+4.2){$f_{0,1}$};

\node[ver] () at (-1.4+7.5,-.5+4.2){$a_{1,0}$};
\node[ver] () at (-1.2+7.5,-1.1+4.2){$b_{1,0}$};
\node[ver] () at (-.6+7.5,-1.6+4.2){$c_{1,0}$};
\node[ver] () at (1.2+7.5,-.5+4.2){$d_{1,0}$};
\node[ver] () at (1.1+7.5,-1.1+4.2){$e_{1,0}$};
\node[ver] () at (.5+7.5,-1.6+4.2){$f_{1,0}$};
\node[ver] () at (1.4+7.5,.5+4.2){$a_{1,1}$};
\node[ver] () at (1.1+7.5,1.1+4.2){$b_{1,1}$};
\node[ver] () at (.5+7.5,1.6+4.2){$c_{1,1}$};
\node[ver] () at (-1.4+7.5,.5+4.2){$d_{1,1}$};
\node[ver] () at (-1.2+7.5,1.1+4.2){$e_{1,1}$};
\node[ver] () at (-.6+7.5,1.6+4.2){$f_{1,1}$};

\node[ver] () at (-1.6+12.5,-.5+4.2){$a_{2,0}$};
\node[ver] () at (-1.2+12.5,-1.1+4){$b_{2,0}$};
\node[ver] () at (-.6+12.5,-1.8+4.2){$c_{2,0}$};
\node[ver] () at (-1.6+12.5,.5+4.2){$d_{2,1}$};
\node[ver] () at (-1.2+12.5,1.1+4.4){$e_{2,1}$};
\node[ver] () at (-.4+12.5,1.6+4.2){$f_{2,1}$};


\node[ver] () at (2-6,-.5+8.3){$d_{-2,2}$};
\node[ver] () at (1.7-6,-1.3+8.5){$e_{-2,2}$};
\node[ver] () at (.1-5,-1.8+8.5){$f_{-2,2}$};
\node[ver] () at (1.2-5,.5+8.3){$a_{-2,3}$};
\node[ver] () at (.8-5,1.3+8.3){$b_{-2,3}$};
\node[ver] () at (.1-5,1.7+8.3){$c_{-2,3}$};

\node[ver] () at (-1.4,-.5+8.3){$a_{-1,2}$};
\node[ver] () at (-1,-1.1+8.3){$b_{-1,2}$};
\node[ver] () at (-.6,-1.6+8.5){$c_{-1,2}$};
\node[ver] () at (1.2,-.5+8.3){$d_{-1,2}$};
\node[ver] () at (.9,-1.1+8.3){$e_{-1,2}$};
\node[ver] () at (.5,-1.6+8.5){$f_{-1,2}$};
\node[ver] () at (1.2,.5+8.3){$a_{-1,3}$};
\node[ver] () at (.9,1.1+8.3){$b_{-1,3}$};
\node[ver] () at (.6,1.5+8.3){$c_{-1,3}$};
\node[ver] () at (-1.4,.5+8.3){$d_{-1,3}$};
\node[ver] () at (-1,1.1+8.3){$e_{-1,3}$};
\node[ver] () at (-.6,1.6+8.3){$f_{-1,3}$};

\node[ver] () at (-1.5+5,-.5+8.3){$a_{0,2}$};
\node[ver] () at (-1.3+5,-1.1+8.3){$b_{0,2}$};
\node[ver] () at (-.6+5,-1.6+8.5){$c_{0,2}$};
\node[ver] () at (1.2+5,-.5+8.3){$d_{0,2}$};
\node[ver] () at (.9+5,-1.1+8.3){$e_{0,2}$};
\node[ver] () at (.5+5,-1.6+8.5){$f_{0,2}$};
\node[ver] () at (1.2+5,.5+8.3){$a_{0,3}$};
\node[ver] () at (.9+5,1.1+8.3){$b_{0,3}$};
\node[ver] () at (.2+5,1.5+8.5){$c_{0,3}$};
\node[ver] () at (-1.6+5,.5+8.3){$d_{0,3}$};
\node[ver] () at (-1.2+5,1.1+8.3){$e_{0,3}$};
\node[ver] () at (-.6+5,1.6+8.4){$f_{0,3}$};

\node[ver] () at (-1.5+10,-.5+8.3){$a_{1,2}$};
\node[ver] () at (-1.3+10,-1.1+8.3){$b_{1,2}$};
\node[ver] () at (-.6+10,-1.6+8.5){$c_{1,2}$};
\node[ver] () at (1.2+10,-.5+8.3){$d_{1,2}$};
\node[ver] () at (.9+10,-1.1+8.3){$e_{1,2}$};
\node[ver] () at (.3+10,-1.6+8.5){$f_{1,2}$};
\node[ver] () at (1.2+10,.5+8.3){$a_{1,3}$};
\node[ver] () at (.9+10,1.1+8.3){$b_{1,3}$};
\node[ver] () at (.2+10,1.5+8.5){$c_{1,3}$};
\node[ver] () at (-1.6+10,.5+8.3){$d_{1,3}$};
\node[ver] () at (-1.2+10,1.1+8.3){$e_{1,3}$};
\node[ver] () at (-.6+10,1.6+8.4){$f_{1,3}$};

\node[ver] () at (12.9,.5){$Q_0$};\node[ver] () at (12.9,2.2){$P_0$};
\node[ver] () at (12.9,6.2){$Q_1$};\node[ver] () at (12.9,7.9){$P_1$};
\node[ver] () at (12.9,8.9){$Q_2$};

\node[ver] () at (1.2,2){$B_{0,0}$};\node[ver] () at (2.4,1.4){$F_{0,0}$};
\node[ver] () at (3.7,2){$C_{0,0}$};\node[ver] () at (5,2.7){$D_{0,0}$};
\node[ver] () at (6.1,2){$B_{1,0}$};\node[ver] () at (7.4,1.4){$F_{1,0}$};
\node[ver] () at (8.7,2){$C_{1,0}$};\node[ver] () at (10,2.7){$D_{1,0}$};

\node[ver] () at (-.1,4.1){$J_{0,0}$};\node[ver] () at (2.3,4.1){$A_{0,0}$};\node[ver] () at (4.8,4.1){$J_{1,0}$};\node[ver] () at (7.2,4.1){$A_{1,0}$};\node[ver] () at (9.8,4.1){$J_{2,0}$};

\node[ver] () at (2.4,7){$D_{-1,1}$};\node[ver] () at (3.7,6.4){$B_{0,1}$};
\node[ver] () at (5,5.6){$F_{0,1}$};\node[ver] () at (6.1,6.4){$C_{0,1}$};
\node[ver] () at (7.4,7){$D_{0,1}$};\node[ver] () at (8.7,6.4){$B_{1,1}$};
\node[ver] () at (9.8,5.6){$F_{1,1}$};\node[ver] () at (11,6.4){$C_{1,1}$};

\node[ver] () at (-.1,8.2){$A_{-1,1}$};
\node[ver] () at (4.8,8.2){$A_{0,1}$};
\node[ver] () at (9.8,8.2){$A_{1,1}$};

\node[ver] () at (3, -0.7){\normalsize \bf Figure 7: Part of {\boldmath $X$} of Lemma \ref{lemma:E10}};
\end{tikzpicture}
\end{figure}

\begin{proof}[Proof of Lemma \ref{lemma:E10}]
Since the type of $X$ is $[4^1,6^1,12^1]$, each 6-gon of $X$ intersects three 12-gons and three 4-gons at edges. Let $c_{0,0}f_{0,0}$ be an edge in the intersection of a 6-gon and a 12-gon. Let these faces be $F_{0,0} := C_{6}(c_{0,0}, f_{0,0}, e_{1,-1}, d_{1,-1}, a_{0,-1}, b_{0,-1})$ and $A_{0,0} := C_{12}(c_{0,0}, f_{0,0}, e_{0,0}, d_{0,0}, a_{0,1}, b_{0,1}, c_{0,1}, f_{0,1}, e_{0,1}, d_{0,1}, a_{0,0}, b_{0,0})$ for some vertices $a_{0,0}$, $a_{0,1}$, $a_{0,-1}$, $b_{0,0}$, $b_{0,1}$, $b_{0,-1}$, $c_{0,1}$, $d_{0,0}$, $d_{0,1}$, $d_{1,-1}$, $e_{1,-1}$, $e_{0,0}$, $e_{0,1}$, $f_{0,1}$ (see Fig. 7). Then the third face through $c_{0,0}$ (resp., $f_{0,0}$) must be of the form $B_{0,0} := C_4(b_{0,0}, c_{0,0}, b_{0,-1}, c_{0,-1})$ (resp., $C_{0,0} := C_{4}(f_{0,0}, e_{0,0}, f_{1,-1}, e_{1,-1})$). Then the third face through $e_{0,0}$ must be of the form $D_{0,0} := C_{6}(e_{0,0}$, $d_{0,0}$, $a_{1,0}, b_{1,0}, c_{1,-1}, f_{1,-1})$. Clearly, the third face through $d_{0,0}$ (resp. $b_{1,0}$) must be of the form $J_{1,0} : = C_{4}(d_{0,0}, a_{1,0}, d_{1,1}, a_{0,1})$ (resp., $B_{1,0} := C_4(b_{1,0}, c_{1,0}, b_{1,-1}, c_{1,-1})$). Then the third face through $c_{1,0}$ must be of the form $F_{1,0} := C_{6}(c_{1,0}, f_{1,0}, e_{2,-1}, d_{2,-1}, a_{1,-1}, b_{1,-1})$.
Continuing this way get the paths (see Fig. 7)
\begin{align*}
  Q_0 & := \cdots\mbox{-}a_{-1,-1}\mbox{-}d_{0,-1}\mbox{-} e_{0,-1}\mbox{-}f_{0,-1}\mbox{-} c_{0,-1}\mbox{-}b_{0,-1}\mbox{-}a_{0,-1}\mbox{-}d_{1,-1} \mbox{-} e_{1,-1}\mbox{-}f_{1,-1} \mbox{-}c_{1,-1}\mbox{-}b_{1,-1}\mbox{-}\cdots \\
  P_0 & := \cdots\mbox{-}c_{-1,0}\mbox{-}f_{-1,0}\mbox{-}e_{-1,0}\mbox{-}d_{-1,0}\mbox{-} a_{0,0} \mbox{-} b_{0,0}\mbox{-}c_{0,0}\mbox{-}f_{0,0}\mbox{-}e_{0,0}\mbox{-}d_{0,0} \mbox{-} a_{1,0}\mbox{-} b_{1,0}\mbox{-}c_{1,0}\mbox{-}f_{1,0}\mbox{-}e_{1,0} \mbox{-}\cdots \\
  Q_1 & := \cdots \mbox{-}b_{-1,1}\mbox{-}a_{-1,1}\mbox{-}d_{0,1}\mbox{-}e_{0,1}\mbox{-} f_{0,1}\mbox{-}c_{0,1}\mbox{-}b_{0,1}\mbox{-}a_{0,1}\mbox{-}d_{1,1} \mbox{-} e_{1,1} \mbox{-} f_{1,1}\mbox{-}c_{1,1}\mbox{-}b_{1,1}\mbox{-}a_{1,1}\mbox{-}\cdots
\end{align*}
and faces $A_{i,0}$, $B_{i,0}$, $C_{i,0}$, $D_{i,0}$, $F_{i,0}$, $J_{i,0}$, $i\in \mathbb{Z}$, where
$J_{i,j} : = C_{4}(d_{i,2j}, a_{i+1,2j}, d_{i+1,2j+1}$, $a_{i,2j+1})$,
$B_{i,j} :=$ $C_4(b_{i,2j}$, $c_{i,2j}$, $b_{i,2j-1}$, $c_{i,2j-1})$,
$C_{i,j} :=$ $C_{4}(f_{i,2j}$, $e_{i,2j}$, $f_{i+1,2j-1}$, $e_{i+1,2j-1})$,
$A_{i,j} :=$ $C_{12}(c_{i,2j}$, $f_{i,2j}$, $e_{i,2j}$, $d_{i,2j}$, $a_{i,2j+1}$, $b_{i,2j+1}$, $c_{i,2j+1}$, $f_{i,2j+1}$, $e_{i,2j+1}$, $d_{i,2j+1}$, $a_{i,2j}$, $b_{i,2j})$,
$D_{i,j} :=$ $C_{6}(e_{i,2j}$, $d_{i,2j}$, $a_{i+1,2j}$, $b_{i+1,2j}$, $c_{i+1,2j-1}$, $f_{i+1, 2j-1})$,
$F_{i,j} :=$ $C_{6}(c_{i,2j}$, $f_{i,2j}$, $e_{i+1,2j-1}$, $d_{i+1,2j-1}$, $a_{i,2j-1}$, $b_{i,2j-1})$.

\medskip

\noindent {\bf Claim.} $P_0$, $Q_0$, $Q_1$ are infinite paths.

\smallskip

If $P_0$ is not a path then we get a cycle (as a subgraph of $P_0$). This gives (by the similar arguments as in the proof of Lemma \ref{lemma:E8}) a map $M$ on the disc $\mathbb{D}^2$ which satisfies either properties (i) and (ii) of Lemma \ref{lemma:disc4612-1} or properties (i) and (ii) of Lemma \ref{lemma:disc4612-2} for some $k, \ell$. But, this is not possible by Lemmas \ref{lemma:disc4612-1}, \ref{lemma:disc4612-2}. Thus, $P_0$ is an infinite path. Similarly, $Q_0$ and $Q_1$ are infinite paths. This proves the claim.

If we start from the edge $c_{0,2}f_{0,2}$ in place of $c_{0,0}f_{0,0}$, we get faces
$C_{12}(c_{0,2}$, $f_{0,2}$, $e_{0,2}$, $d_{0,2}$, $a_{0,3}$, $b_{0,3}$, $c_{0,3}$, $f_{0,3}$, $e_{0,3}$, $d_{0,3}$, $a_{0,2}$, $b_{0,2})$, $C_4(b_{0,2}, c_{0,2}, b_{0,1}, c_{0,1})$, $C_{4}(f_{0,2}, e_{0,2}, f_{1,1}, e_{1,1})$ (in place of $A_{0,0}, B_{0,0}$, $C_{0,0}$ respectively),  for some vertices $a_{0,2}$, $a_{0,3}$, $b_{0,2}$, $b_{0,3}$, $c_{0,3}$, $d_{0,2}$, $d_{0,3}$, $e_{0,2}$, $e_{0,3}$, $f_{0,3}$ (see Fig. 7). Then, by the same method as above, we get faces $A_{i,1}$, $B_{i,1}$, $C_{i,1}$, $D_{i,1}$, $F_{i,1}$, $J_{i,1}$, $i\in \mathbb{Z}$. Continuing this way we faces
$A_{i,j}$, $B_{i,j}$, $C_{i,j}$, $D_{i,j}$, $F_{i,j}$, $J_{i,j}$, for $i, j\in \mathbb{Z}$, of $X$.
Then the mapping $\varphi \colon X \mapsto E_{10}$, given by $\varphi(a_{i,j}) = u_{i,j}$,  $\varphi(b_{i,j}) = v_{i,j}$, $\varphi(c_{i,j}) = w_{i,j}$, $\varphi(d_{i,j}) = x_{i,j}$,  $\varphi(e_{i,j}) = y_{i,j}$, $\varphi(c_{i,j}) = z_{i,j}$, is an isomorphism.
\end{proof}

\begin{lemma} \label{lemma:E11}
Let $E_{11}$ be as in Example $\ref{exam:plane}$.
Let $X$ be a semi-equivelar map on the plane. If the type of $X$ is $[4^1, 8^2]$ then $X\cong E_{11}$.
\end{lemma}

To prove Lemma \ref{lemma:E11}, we need the following two technical lemmas.

\begin{lemma} \label{lemma:disc482-1}
Let $M$ be a map on the $2$-disk $\mathbb{D}^2$ whose faces are quadrangles and $8$-gons. For a vertex $u$ of $M$ and $i=4, 8$, let $n_i(u)$ be the number of $i$-gons through $u$.  Then $M$ can not satisfy both the following: {\rm (i)} $(n_4(u), n_{8}(u)) = (1, 2)$ for each internal vertex $u$, {\rm (ii)} boundary $\partial \mathbb{D}^2$ has $k+2\ell$ vertices in which $k$ vertices are in two $8$-gons, and $2\ell$ vertices which are in one $8$-gon only, for some $k > 0$ and $k \le \ell \le k+3$.
\end{lemma}

\begin{proof}
Let $f_0, f_1$ and $f_2$ be the number of vertices, edges and faces of $M$ respectively. For $i= 4, 8$, let $n_i$ be the total number of $i$-gons in $M$.
Let there be $n$ internal vertices. Then $f_0=n+k+2\ell$   and $f_1= (3\times n+3\times k+2\times 2\ell)/2= 2\ell+3(n+k)/2$. So, $n+k$ is even.

Suppose $M$ satisfies (i) and (ii).
Then $n_4=n/4$. So, $n$ is even and (since $n+k$ is even) hence $k$ is even, say $k=2r$.
Now, $n_8=(2n+2k+2\ell)/8= (n+k+\ell)/4$. These imply that $\ell$ is also even and hence $\ell = k$ or $k+2$.
If $\ell = k+2$ then $n_8 = n/4 + (2k+2)/4 = n_4 + r +1/2\not\in\mathbb{Z}$, a contradiction. Thus, $\ell = k$. Therefore $f_0 = n+3k$,  $f_1= (3n+7k)/2$
and $f_2= n_4+n_8 = n/4+ (n+2k)/4= (n+k)/2$.
These imply, $f_0-f_1+f_2 = 0$. This is not possible since the Euler characteristic of the 2-disk $\mathbb{D}^2$ is 1. This completes the proof.
\end{proof}


\begin{lemma} \label{lemma:disc482-2}
Let $M$ be a map on the $2$-disk $\mathbb{D}^2$ whose faces are quadrangles and $8$-gons. For a vertex $u$ of $M$ and $i=4,8$, let $n_i(u)$ be the number of $i$-gons through $u$. Then $M$ can not satisfy both the following: {\rm (i)} $(n_4(u), n_{8}(u)) = (1, 2)$ for each internal vertex $u$, {\rm (ii)} boundary $\partial \mathbb{D}^2$ has $k+2\ell$ vertices in which $k$ vertices are in one $4$-gon, and $2\ell$ vertices are in one $4$-gon and $8$-gon, for some $k > 0$ and $k \le \ell \le k+2$.
\end{lemma}

\begin{proof}
Let $f_0, f_1$ and $f_2$ denote the number of vertices, edges and faces of $M$ respectively. For $i=4, 8$, let $n_i$ be the total number of $i$-gons in $M$. Let there be $n$ internal vertices. Then $f_0=n+k+2\ell$ and $f_1=(3\times n+2\times k+ 3\times 2\ell)/2$. So, $n$ is even, say $n=2n^{\prime}$.

Suppose $M$ satisfies (i) and (ii).
Then $n_4=(n+k+2\ell)/4= (2n^{\prime}+k+2\ell)/4$ and $n_8=(2n+2\ell)/8 = (2n^{\prime}+\ell)/4$.
Since $n_4, n_8\in \mathbb{Z}$, it follows that both $k$ and $\ell$ are even. Let $k=2k^{\prime}$ and $\ell = 2\ell^{\prime}$.
Since $k \le \ell \le k+2$, it follows that $\ell = k, k+2$.
If $\ell = k+2$, then $n_4 - n_8 =  (k+\ell)/4 = (2k+2)/4 = k^{\prime}+1/2 \not\in\mathbb{Z}$, a contradiction.  So, $\ell = k$.
Then $f_0 = n+3k$, $f_1= (3n+8k)/2$ and $f_2=n_4+n_8=(n+3k)/4+ (2n+2k)/8=(n+2k)/2$. These imply, $f_0-f_1+f_2 = 0$. This is not possible since the Euler characteristic of the 2-disk $\mathbb{D}^2$ is 1. This completes the proof.
\end{proof}

\begin{figure}[ht]
\tiny
\tikzstyle{ver}=[]
\tikzstyle{vert}=[circle, draw, fill=black!100, inner sep=0pt, minimum width=4pt]
\tikzstyle{vertex}=[circle, draw, fill=black!00, inner sep=0pt, minimum width=4pt]
\tikzstyle{edge} = [draw,thick,-]
\centering

\begin{tikzpicture}[scale=0.45]

\draw ({-3.7+2*cos(22.5)}, {2*sin(22.5)}) -- ({-3.7+2*cos(67.5)}, {2*sin(67.5)}) -- ({-5.7+2*cos(22.5)}, {2*sin(67.5)});

\draw ({2*cos(22.5)}, {2*sin(22.5)}) -- ({2*cos(67.5)}, {2*sin(67.5)}) -- ({2*cos(112.5)}, {2*sin(112.5)}) -- ({2*cos(157.5)}, {2*sin(157.5)});

\draw ({3.7+2*cos(22.5)}, {2*sin(22.5)}) -- ({3.7+2*cos(67.5)}, {2*sin(67.5)}) -- ({3.7+2*cos(112.5)}, {2*sin(112.5)}) -- ({3.7+2*cos(157.5)}, {2*sin(157.5)});

\draw ({7.4+2*cos(22.5)}, {2*sin(22.5)}) -- ({7.4+2*cos(67.5)}, {2*sin(67.5)}) -- ({7.4+2*cos(112.5)}, {2*sin(112.5)}) -- ({7.4+2*cos(157.5)}, {2*sin(157.5)});

\draw ({10.4+2*cos(67.5)}, {2*sin(67.5)}) -- ({11.1+2*cos(112.5)}, {2*sin(112.5)}) -- ({11.1+2*cos(157.5)}, {2*sin(157.5)});

\draw [fill = lightgray] ({11.1+2*cos(157.5)}, {2*sin(157.5)}) -- ({7.4+2*cos(67.5)}, {2*sin(67.5)}) -- (9.2,3) -- ({11.1+2*cos(112.5)}, {2*sin(112.5)}) -- ({11.1+2*cos(157.5)}, {2*sin(157.5)});

\draw [xshift = -105.5, fill = lightgray] ({11.1+2*cos(157.5)}, {2*sin(157.5)}) -- ({7.4+2*cos(67.5)}, {2*sin(67.5)}) -- (9.2,3) -- ({11.1+2*cos(112.5)}, {2*sin(112.5)}) -- ({11.1+2*cos(157.5)}, {2*sin(157.5)});

\draw [xshift = -211, fill = lightgray] ({11.1+2*cos(157.5)}, {2*sin(157.5)}) -- ({7.4+2*cos(67.5)}, {2*sin(67.5)}) -- (9.2,3) -- ({11.1+2*cos(112.5)}, {2*sin(112.5)}) -- ({11.1+2*cos(157.5)}, {2*sin(157.5)});

\draw [xshift = -316.5, fill = lightgray] ({11.1+2*cos(157.5)}, {2*sin(157.5)}) -- ({7.4+2*cos(67.5)}, {2*sin(67.5)}) -- (9.2,3) -- ({11.1+2*cos(112.5)}, {2*sin(112.5)}) -- ({11.1+2*cos(157.5)}, {2*sin(157.5)});


\draw [yshift = 106,fill = lightgray] ({-3+2*cos(247.5)}, {2*sin(247.5)}) -- ({-3.7+2*cos(292.5)}, {2*sin(292.5)}) -- ({-3.7+2*cos(337.5)}, {2*sin(337.5)}) -- ({-3.7+2*cos(22.5)}, {2*sin(22.5)}) -- ({-3.7+2*cos(67.5)}, {2*sin(67.5)}) -- ({-5.7+2*cos(22.5)}, {2*sin(67.5)});

\draw [yshift = 106,fill = lightgray] ({2*cos(22.5)}, {2*sin(22.5)}) -- ({2*cos(67.5)}, {2*sin(67.5)}) -- ({2*cos(112.5)}, {2*sin(112.5)}) -- ({2*cos(157.5)}, {2*sin(157.5)}) -- ({2*cos(202.5)}, {2*sin(202.5)}) -- ({2*cos(247.5)}, {2*sin(247.5)}) -- ({2*cos(292.5)}, {2*sin(292.5)}) -- ({2*cos(337.5)}, {2*sin(337.5)}) -- ({2*cos(22.5)}, {2*sin(22.5)});

\draw [yshift = 106,fill = lightgray] ({3.7+2*cos(22.5)}, {2*sin(22.5)}) -- ({3.7+2*cos(67.5)}, {2*sin(67.5)}) -- ({3.7+2*cos(112.5)}, {2*sin(112.5)}) -- ({3.7+2*cos(157.5)}, {2*sin(157.5)}) -- ({3.7+2*cos(202.5)}, {2*sin(202.5)}) -- ({3.7+2*cos(247.5)}, {2*sin(247.5)}) -- ({3.7+2*cos(292.5)}, {2*sin(292.5)}) -- ({3.7+2*cos(337.5)}, {2*sin(337.5)}) -- ({3.7+2*cos(22.5)}, {2*sin(22.5)});

\draw [yshift = 106,fill = lightgray] ({7.4+2*cos(22.5)}, {2*sin(22.5)}) -- ({7.4+2*cos(67.5)}, {2*sin(67.5)}) -- ({7.4+2*cos(112.5)}, {2*sin(112.5)}) -- ({7.4+2*cos(157.5)}, {2*sin(157.5)}) -- ({7.4+2*cos(202.5)}, {2*sin(202.5)}) -- ({7.4+2*cos(247.5)}, {2*sin(247.5)}) -- ({7.4+2*cos(292.5)}, {2*sin(292.5)}) -- ({7.4+2*cos(337.5)}, {2*sin(337.5)}) -- ({7.4+2*cos(22.5)}, {2*sin(22.5)});

\draw [yshift = 106,fill = lightgray] ({10.4+2*cos(67.5)}, {2*sin(67.5)}) -- ({11.1+2*cos(112.5)}, {2*sin(112.5)}) -- ({11.1+2*cos(157.5)}, {2*sin(157.5)}) -- ({11.1+2*cos(202.5)}, {2*sin(202.5)}) -- ({11.1+2*cos(247.5)}, {2*sin(247.5)}) -- ({10.4+2*cos(292.5)}, {2*sin(292.5)});

\draw [yshift = 212] ({-3+2*cos(247.5)}, {2*sin(247.5)}) -- ({-3.7+2*cos(292.5)}, {2*sin(292.5)}) -- ({-3.7+2*cos(337.5)}, {2*sin(337.5)}) -- ({-3.7+2*cos(22.5)}, {2*sin(22.5)}) -- ({-3.7+2*cos(67.5)}, {2*sin(67.5)}) -- ({-5.7+2*cos(22.5)}, {2*sin(67.5)});

\draw [yshift = 212] ({2*cos(22.5)}, {2*sin(22.5)}) -- ({2*cos(67.5)}, {2*sin(67.5)}) -- ({2*cos(112.5)}, {2*sin(112.5)}) -- ({2*cos(157.5)}, {2*sin(157.5)}) -- ({2*cos(202.5)}, {2*sin(202.5)}) -- ({2*cos(247.5)}, {2*sin(247.5)}) -- ({2*cos(292.5)}, {2*sin(292.5)}) -- ({2*cos(337.5)}, {2*sin(337.5)}) -- ({2*cos(22.5)}, {2*sin(22.5)});

\draw [yshift = 212] ({3.7+2*cos(22.5)}, {2*sin(22.5)}) -- ({3.7+2*cos(67.5)}, {2*sin(67.5)}) -- ({3.7+2*cos(112.5)}, {2*sin(112.5)}) -- ({3.7+2*cos(157.5)}, {2*sin(157.5)}) -- ({3.7+2*cos(202.5)}, {2*sin(202.5)}) -- ({3.7+2*cos(247.5)}, {2*sin(247.5)}) -- ({3.7+2*cos(292.5)}, {2*sin(292.5)}) -- ({3.7+2*cos(337.5)}, {2*sin(337.5)}) -- ({3.7+2*cos(22.5)}, {2*sin(22.5)});

\draw [yshift = 212] ({7.4+2*cos(22.5)}, {2*sin(22.5)}) -- ({7.4+2*cos(67.5)}, {2*sin(67.5)}) -- ({7.4+2*cos(112.5)}, {2*sin(112.5)}) -- ({7.4+2*cos(157.5)}, {2*sin(157.5)}) -- ({7.4+2*cos(202.5)}, {2*sin(202.5)}) -- ({7.4+2*cos(247.5)}, {2*sin(247.5)}) -- ({7.4+2*cos(292.5)}, {2*sin(292.5)}) -- ({7.4+2*cos(337.5)}, {2*sin(337.5)}) -- ({7.4+2*cos(22.5)}, {2*sin(22.5)});

\draw [yshift = 212] ({10.4+2*cos(67.5)}, {2*sin(67.5)}) -- ({11.1+2*cos(112.5)}, {2*sin(112.5)}) -- ({11.1+2*cos(157.5)}, {2*sin(157.5)}) -- ({11.1+2*cos(202.5)}, {2*sin(202.5)}) -- ({11.1+2*cos(247.5)}, {2*sin(247.5)}) -- ({10.4+2*cos(292.5)}, {2*sin(292.5)});


\draw ({3.7+2*cos(22.5)}, {2*sin(22.5)}) -- ({3.7+2*cos(67.5)}, {2*sin(67.5)}) -- ({3.7+2*cos(112.5)}, {2*sin(112.5)}) -- ({3.7+2*cos(157.5)}, {2*sin(157.5)}) -- ({3.7+2*cos(202.5)}, {2*sin(202.5)}) -- ({3.7+2*cos(247.5)}, {2*sin(247.5)}) -- ({3.7+2*cos(292.5)}, {2*sin(292.5)}) -- ({3.7+2*cos(337.5)}, {2*sin(337.5)}) -- ({3.7+2*cos(22.5)}, {2*sin(22.5)});


\node[ver] () at (-3.3,-1.5+3){\scriptsize $b_{-3,0}$};
\node[ver] () at (-.2,-1.5+3.6){$b_{-2,0}$};
\node[ver] () at (.5,-1.5+3){$b_{-1,0}$};
\node[ver] () at (3.4,-1.5+3.6){$b_{0,0}$}; \node[ver] () at (3.4,-1.4){$b_{0,-1}$};
\node[ver] () at (4.1,-1.5+3){$b_{1,0}$}; \node[ver] () at (5.8,-1.6){$b_{1,-1}$};
\node[ver] () at (7.1,-1.5+3.6){$b_{2,0}$};
\node[ver] () at (7.9,-1.5+3){$b_{3,0}$};
\node[ver] () at (-.6+11.2,-1.5+3.6){$b_{4,0}$};

\node[ver] () at (-3.3,5.9){$b_{-3,1}$};
\node[ver] () at (-.2,1.5+3.7){$b_{-2,1}$};
\node[ver] () at (.4,5.9){$b_{-1,1}$};
\node[ver] () at (3.4,1.5+3.7){$b_{0,1}$};
\node[ver] () at (.4+3.7,5.9){$b_{1,1}$};
\node[ver] () at (7.1,1.5+3.7){$b_{2,1}$};
\node[ver] () at (.4+7.4,5.9){$b_{3,1}$};
\node[ver] () at (-.5+11.1,1.5+3.7){$b_{4,1}$};

\node[ver] () at (-2.7,9.6){$b_{-3,2}$};
\node[ver] () at (-.4,9.6){$b_{-2,2}$};
\node[ver] () at (1.2,9.6){$b_{-1,2}$};
\node[ver] () at (3.4,9.6){$b_{0,2}$};
\node[ver] () at (4.4,9.6){$b_{1,2}$};
\node[ver] () at (-.5+7.4,9.6){$b_{2,2}$};
\node[ver] () at (8.2,9.6){$b_{3,2}$};
\node[ver] () at (-.5+11.1,9.6){$b_{4,2}$};


\node[ver] () at (-1+10.8,.8+3.6+3.7){$a_{2,3}$};
\node[ver] () at (-1+10.8,3+3.7){$a_{2,2}$};
\node[ver] () at (-1+10.8,.8+3.6){$a_{2,1}$};
\node[ver] () at (-1+10.8,3){$a_{2,0}$};
\node[ver] () at (-1+11.2,.8+3.6-3.7){$a_{2,-1}$};

\node[ver] () at (-1+7.2,.8+3.6+3.7){$a_{1,3}$};
\node[ver] () at (-1+7.2,3+3.7){$a_{1,2}$};
\node[ver] () at (-1+7.2,.8+3.6){$a_{1,1}$};
\node[ver] () at (-1+7.2,3){$a_{1,0}$};\node[ver] () at (-1+7.4,-.7){$a_{1,-2}$};
\node[ver] () at (-1+7.4,.8+3.6-3.7){$a_{1,-1}$};

\node[ver] () at (-1+3.6,.8+3.6+3.7){$a_{0,3}$};
\node[ver] () at (-1+3.6,3+3.7){$a_{0,2}$};
\node[ver] () at (-1+3.6,.8+3.6){$a_{0,1}$};
\node[ver] () at (-1+3.6,3){$a_{0,0}$};\node[ver] () at (-1+3.6,-.7){$a_{0,-2}$};
\node[ver] () at (-1+3.7,.8+3.6-3.7){$a_{0,-1}$};

\node[ver] () at (-1,.8+3.6+3.7){$a_{-2,3}$};
\node[ver] () at (-1,3+3.7){$a_{-2,2}$};
\node[ver] () at (-1,.8+3.6){$a_{-2,1}$};
\node[ver] () at (-1,3){$a_{-2,0}$};
\node[ver] () at (-.8,.8+3.6-3.7){$a_{-2,-1}$};

\node[ver] () at (12,9.3){$P_2$};\node[ver] () at (12,5.5){$P_1$};\node[ver] () at (12,1.8){$P_0$};

\node[ver] () at (4,7.3){$A_{0,1}$};\node[ver] () at (7.5,7.3){$A_{1,1}$};

\node[ver] () at (4,3.5){$A_{0,0}$};\node[ver] () at (7.5,3.5){$A_{1,0}$};\node[ver] () at (.5,3.5){$A_{-1,0}$};\node[ver] () at (11,3.5){$A_{2,0}$};\node[ver] () at (4,-.2){$A_{0,-1}$};

\node[ver] () at (2,1.7){$B_{0,0}$};\node[ver] () at (5.7,1.7){$B_{1,0}$};\node[ver] () at (9.3,1.7){$B_{2,0}$};

\node[ver] () at (2,5.6){$B_{0,1}$};\node[ver] () at (5.7,5.6){$B_{1,1}$};\node[ver] () at (9.3,5.6){$B_{2,1}$};

\node[ver] () at (4, -3.2){\normalsize \bf Figure 8: Part of {\boldmath $X$} of Lemma \ref{lemma:E11}};

\end{tikzpicture}
\end{figure}

\vspace{-5mm}

\begin{proof}[Proof of Lemma \ref{lemma:E11}]
Let $b_{0,0}b_{1,0}$ be an edge of $X$ in two 8-gons. Assume that these 8-gons are $A_{0,-1} :=$ $C_{8}(b_{0,0}, b_{1,0}, a_{1,-1}, a_{1,-2}, b_{1,-1},b_{0,-1}, a_{0,-2}, a_{0,-1})$ and $A_{0,0} :=$ $C_{8}(b_{0,0}$, $b_{1,0}$, $a_{1,0}$, $a_{1,1}$, $b_{1,1}$, $b_{0,1}$, $a_{0,1}$, $a_{0,0,})$, for some vertices $a_{0,0}$, $a_{1,0}$, $a_{0,1}$, $a_{1,1}$, $a_{1,-1}$, $a_{1,-2}$, $a_{0,-2}$, $b_{0,1}$, $b_{1,1}$, $b_{0,-1}$, $b_{1,-1}$. Then the third face through $b_{0,0}$ (resp. $b_{1,0}$) must be of the form $B_{0,0} :=$ $C_4(b_{-1,0}$, $a_{0,-1}$, $b_{0,0}$, $a_{0,0})$ (resp. $B_{1,0} :=$ $C_4(b_{1,0}$, $a_{1,-1}$, $b_{2,0}$, $a_{1,0})$). Clearly, the second face through the edge $a_{1,0}a_{1,1}$ must be of the form $A_{1,0} :=$ $C_{8}(b_{2,0}$, $b_{3,0}$, $a_{2,0}$, $a_{2,1}$, $b_{3,1}$, $b_{2,1}$, $a_{1,1}$, $a_{1,0,})$. Now, the third face through $a_{1,1}$ must be of the form $B_{1,1} :=$ $C_4(b_{1,1}$, $a_{1,1}$, $b_{2,1}$, $a_{1,2})$. Continuing this way (by the similar arguments as in the proofs of Lemmas \ref{lemma:E8}, \ref{lemma:E9} and \ref{lemma:E10}) we get the paths $P_0 := \cdots\mbox{-} b_{-2,0}\mbox{-} b_{-1,0}\mbox{-} a_{0,-1}\mbox{-}b_{0,0}\mbox{-}b_{1,0}\mbox{-}a_{1,-1}\mbox{-} b_{2,0} \mbox{-} \cdots$ and $P_1 := \cdots\mbox{-} b_{-2,1}\mbox{-} b_{-1,1}\mbox{-} a_{0,1}\mbox{-}b_{0,1}\mbox{-}b_{1,1}\mbox{-}a_{1,1}\mbox{-} b_{2,1} \mbox{-} \cdots$ and faces $A_{i,0}$, $B_{i,0}$, $i\in \mathbb{Z}$, where
$A_{i,j} :=$ $C_{8}(b_{2i,j}, b_{2i+1,j}$, $a_{i+1,2j}, a_{i+1,2j+1}$, $b_{2i+1,j+1}, b_{2i, j+1}$, $a_{i,2j+1}, a_{i,2j})$,
$B_{i,j} :=$ $C_{4}(b_{2i-1,j}, a_{i,2j-1}$, $b_{2i,j}, a_{i,2j})$. Again, by Lemmas \ref{lemma:disc482-1}, \ref{lemma:disc482-2}, $P_0$ and $P_1$ are infinite paths (as before). Similarly, if we start with the edge $b_{0,1}b_{1,1}$ in place of $b_{0,0}b_{1,0}$, we get (similar arguments as before) infinite paths $P_1$ and $P_2$ and faces $A_{i,1}$, $B_{i,1}$, $i\in \mathbb{Z}$. Continuing this way we get faces $A_{i,j}$, $B_{i,j}$, $i, j\in \mathbb{Z}$. Then the mapping $\varphi : X \to E_{11}$, given by $\varphi(a_{i,j}) = u_{i,j}$, $\varphi(b_{i,j}) = v_{i,j}$, $i, j\in \mathbb{Z}$, is an isomorphism. This completes the proof.
\end{proof}

\begin{proof}[Proof of Theorem \ref{theo:plane}]
Let $X$ be a semi-equivelar map on $\mathbb{R}^2$. If the type of $X$ is $[3^6]$ then, by \cite[Lemma 3.2]{DU2005}, $X \cong E_1$.
If the type of $X$ is $[3^3,4^2]$ then, by \cite[Lemma 3.4]{DM2017}, $X\cong E_4$.
If the type of $X$ is $[3^2,4^1,3^1,4^1]$ (resp., $[3^1,6^1,3^1,6^1]$) then, by \cite[Lemma 5.3]{DM2017}, $X$ is isomorphic to $E_5$ (resp., $E_6$). If the type of $X$ is $[4^4]$ then by the same arguments (as in the proofs of previous three cases) it is easy to see that $X\cong E_2$.
If the type of $X$ is $[6^3]$ then the dual $X^{\ast}$ of $X$ is a semi-equivelar map  of type $[3^6]$ on $\mathbb{R}^2$. Thus, $X^{\ast} \cong E_1$ and hence $X \cong E_1^{\ast}\cong E_3$. The result now follows by Lemmas \ref{lemma:E7}, \ref{lemma:E8}, \ref{lemma:E9}, \ref{lemma:E10} and \ref{lemma:E11}.
\end{proof}

\begin{proof}[Proof of Theorem \ref{theo:sem=archimedian}]
Let $\mathcal{T} := \{[3^6]$, $[6^3]$, $[4^4]$, $[3^4,6^1]$, $[3^3,4^2]$, $[3^2,4^1,3^1,4^1]$,  $[3^1,6^1,3^1,6^1]$, $[3^1,4^1,6^1,4^1]$,  $[3^1,12^2]$, $[4^1,8^2]$, $[4^1,6^1,12^1]\}$ and $\mathcal{K} := \mathcal{T} \setminus\{[3^4,6^1]\}$.

Let $Y$ be a semi-equivelar map on the torus.
Let the type of $Y$ be $[p_1^{n_1}, \dots, p_k^{n_k}]$.
Then, by \cite[Theorem 1.4]{DM2017}, $[p_1^{n_1}, \dots, p_k^{n_k}] \in \mathcal{T}$.
Since $\mathbb{R}^2$ is the
universal cover of the torus, by pulling back $Y$ (using similar arguments as in the proof of Theorem 3 in \cite[Page 144]{Sp1966}), we get a semi-equivelar map $\widetilde{Y}$ of type $[p_1^{n_1}, \dots, p_k^{n_k}]$ on
$\mathbb{R}^2$ and a polyhedral covering map $\eta \colon \widetilde{Y} \to Y$.
Since $[p_1^{n_1}, \dots, p_k^{n_k}]\in \mathcal{T}$, because of Theorem \ref{theo:plane}, we can assume that $\widetilde{Y}$ is $E_i$ for some $i = 1, 2, \dots, 10$ or 11.
Let $\Gamma_i$ be the group of covering transformations. Then $|Y|= |E_i|/\Gamma_i$.
For $\sigma\in \Gamma_i$, $\eta \circ \sigma = \eta$.
So, $\sigma$ maps a face of the map $E_i$ in $\mathbb{R}^2$ to a face of $E_i$ (in $\mathbb{R}^2$). This implies that $\sigma$
induces an automorphism $\widehat{\sigma}$ (say) of $E_i$. Thus, we can identify
$\Gamma_i$ with a subgroup of ${\rm Aut}(E_i)$. So, $Y$ is a quotient
of $E_i$ by the subgroup $\Gamma_i$ of ${\rm Aut}(E_i)$, where $\Gamma_i$
has no fixed element (vertex, edge or face).

If $Z$ is a semi-equivelar map on the Klein bottle then, by \cite[Theorem 1.4]{DM2017}, the type of $Z$ is in $\mathcal{K}$. Now, by the same argument as above, $Z \cong E_j/\Lambda_j$ for some (fixed element free) subgroup $\Lambda_j$ of ${\rm Aut}(E_j)$ and $j= 1, \dots, 6, 8, 9, 10$ or 11. This completes the proof.
\end{proof}


\section{Proofs of Theorems \ref{theo:no-of-orbits} and \ref{theo:orbits-3&6}} \label{sec:proofs-2}

\begin{proof}[Proof of Theorem \ref{theo:no-of-orbits}]
For $1\leq i\leq 11$, let $E_{i}$ be as in Example \ref{exam:plane}.
Let $V_{i} = V(E_i)$ be the vertex set of $E_{i}$. Let $H_{i}$ be the group of all the translations of $E_{i}$. So, $H_i \leq {\rm Aut}(E_{i})$.

Let $X$ be a semi-equivelar map of type $[4^1,8^2]$ on the torus.
Then, by Theorems \ref{theo:plane} and \ref{theo:sem=archimedian}, we can assume $X = E_{11}/\Gamma_{11}$ for some subgroup $\Gamma_{11}$ of ${\rm Aut}(E_{11})$ and $\Gamma_{11}$
has no fixed element (vertex, edge or face). Hence $\Gamma_{11}$
consists of translations and glide reflections. Since $X =
E_{11}/\Gamma_{11}$ is orientable, $\Gamma_{11}$ does not contain any glide
reflection. Thus $\Gamma_{11} \leq H_{11}$. Let $\eta_{11} : E_{11}\to X$ be the polyhedral covering map.

We take origin $(0,0)$ is the middle point of the line segment joining vertices $v_{0,0}$ and $v_{1,1}$ of $E_{11}$ (see Fig. 1 (k)). Let $a := u_{1,0} - u_{0,0}$, $b := v_{0,1}- v_{0,0} \in \mathbb{R}^2$. Then $H_{11} := \langle x\mapsto x+a, x\mapsto x+b \rangle$. Under the action of $H_{11}$, vertices of $E_{11}$ form four orbits.
Consider the subgroup $G_{11}$ of ${\rm Aut}(E_{11})$ generated by $H_{11}$  and the map (the half rotation) $x\mapsto -x$. So,
\begin{align*}
  G_{11} & =\{ \alpha : x\mapsto \varepsilon x + ma + nb \, : \, \varepsilon=\pm 1, m, n \in \ZZ\} \cong H_{11}\rtimes \mathbb{Z}_2.
\end{align*}
Clearly, under the action of $G_{11}$, vertices of $E_{11}$ form two orbits. The orbits are $O_1 :=\{u_{i,j} \, : \, i, j\in \ZZ\}$, $O_2 := \{v_{i,j} \, : \, i, j\in \ZZ\}$.

\medskip

\noindent {\bf Claim 1.} If $K \leq H_{11}$ then $K \unlhd G_{11}$.

\smallskip

Let $g \in G_{11}$ and $k\in K$. Then $g(x) = \varepsilon x+ma+nb$ and $k(x) = x + pa+ qb$ for some $m, n, p, q\in \mathbb{Z}$ and $\varepsilon\in\{1, -1\}$.
Therefore, $(g\circ k\circ g^{-1})(x) = (g\circ k)(\varepsilon(x-ma-nb)) = g(\varepsilon(x-ma-nb)+pa+qb)=x-ma-nb+\varepsilon(pa+qb)+ma+nb= x+\varepsilon(pa+qb) = k^{\varepsilon}(x)$. Thus, $g\circ k\circ g^{-1} = k^{\varepsilon}\in K$. This proves the claim.

\smallskip

By Claim 1, $\Gamma_{11}$ is a normal subgroup of $G_{11}$. Therefore, $G_{11}/\Gamma_{11}$ acts on $X= E_{11}/\Gamma_{11}$.
Since $O_1$ and $O_2$ are the $G_{11}$-orbits, it follows that $\eta_{11}(O_1)$ and $\eta_{11}(O_2)$ are the $(G_{11}/\Gamma_{11})$-orbits. Since the vertex set of $X$ is $\eta_{11}(V_{11}) = \eta_{11}(O_1) \sqcup \eta_{11}(O_2)$ and $G_{11}/\Gamma_{11} \leq {\rm Aut}(X)$, it follows that the number of ${\rm Aut}(X)$-orbits of vertices is $\leq 2$.
This proves part (a).



\medskip

Now, let $X$ be a semi-equivelar map of type $[3^4,6^1]$ on the torus.
Then, by Theorems \ref{theo:plane} and \ref{theo:sem=archimedian}, we can assume $X = E_{7}/\Gamma_{7}$ for some $\Gamma_{7} \leq {\rm Aut}(E_{7})$. As before, $\Gamma_{7} \leq H_{7}$. Let $\eta_{7} : E_{7}\to X$ be the polyhedral covering map.

We take origin $(0,0)$ is the middle point of the line segment joining vertices $u_{-1,0}$ and $u_{1,0}$ of $E_{7}$ (see Fig. 1 (g)). Let $a_7 := u_{4,-1} - u_{1,0}$, $b_7 := u_{2,2} - u_{1,0}$, $c_7 := u_{-1,3} - u_{1,0} \in \mathbb{R}^2$. Then $H_{7} = \langle x\mapsto x+a_7, x\mapsto x+b_7, x\mapsto x+c_7 \rangle$. Under the action of $H_{7}$, vertices of $E_{7}$ form six orbits. (Observe that the six vertices of a 6-gon are in different $H_7$-orbits.)
Consider the subgroup $G_{7}$ of ${\rm Aut}(E_{7})$ generated by $H_{7}$  and the mapping   $x\mapsto -x$. So,
\begin{align*}
  G_{7} & =\{ \alpha : x\mapsto \varepsilon x + ma_7 + nb_7 + pc_7 \, : \, \varepsilon=\pm 1, m, n, p \in \ZZ\} \cong H_{7}\rtimes \mathbb{Z}_2.
\end{align*}
Then, the vertices of $E_{7}$ form three $G_{7}$-orbits. The orbits are $O_1 :=\{u_{7i+4j\pm 1,j} \, : \, i, j\in \ZZ\}$, $O_2 := \{u_{7i+4j,j\pm 1} \, : \, i, j\in \ZZ\}$ and $O_3 :=\{u_{7i+4j+1,j-1}, u_{7i+4j-1, j+1} \, : \, i, j\in \ZZ\}$. (Observe that each vertex of $E_7$ is in a unique 6-gon and the vertices of the 6-gon whose center is $u_{7i+4j,j}$ are $u_{7i+4j\pm 1,j}$, $u_{7i+4j,j\pm 1}$, $u_{7i+4j+1,j-1}$, $u_{7i+4j-1, j+1}$ for all $i, j$, see Fig. 1 (a), (g).)

\medskip

\noindent {\bf Claim 2.} If $K \leq H_{7}$ then $K \unlhd G_{7}$.

\smallskip

Let $g \in G_{7}$ and $k\in K$. Then, $g(x) = \varepsilon x+ma_7+nb_7+pc_7$ and $k(x) = x + ra_7+ sb_7+tc_7$ for some $m, n, p, r, s, t\in \mathbb{Z}$ and $\varepsilon\in\{1, -1\}$.
Therefore, $(g\circ k\circ g^{-1})(x) = (g\circ k)(\varepsilon(x-ma_7-nb_7-pc_7)) = g(\varepsilon(x-ma_7-nb_7-pc_7)+ra_7+sb_7+tc_7)=x-ma_7-nb_7-pc_7+\varepsilon(ra_7+sb_7+tc_7) + ma_7+nb_7+pc_7= x+\varepsilon(ra_7+ sb_7+tc_7) = k^{\varepsilon}(x)$. Thus, $g\circ k\circ g^{-1} = k^{\varepsilon}\in K$. This proves the claim.

\medskip

By Claim 2, $\Gamma_{7}$ is a normal subgroup of $G_{7}$. Therefore, $G_{7}/\Gamma_{7}$ acts on $X= E_{7}/\Gamma_{7}$.
Since $O_1$, $O_2$ and $O_3$ are the $G_{7}$-orbits, it follows that $\eta_{7}(O_1)$, $\eta_{7}(O_2)$ and $\eta_{7}(O_3)$ are the $(G_{7}/\Gamma_{7})$-orbits. Since the vertex set of $X$ is $\eta_{7}(V_{7}) = \eta_{7}(O_1) \sqcup \eta_{7}(O_2)\sqcup \eta_{7}(O_3)$ and $G_{7}/\Gamma_{7} \leq {\rm Aut}(X)$, it follows that the number of ${\rm Aut}(X)$-orbits of vertices is $\leq 3$.

If $X$ is a semi-equivelar map of type $[3^1, 4^1, 6^1, 4^1]$ on the torus then, by Theorems \ref{theo:plane} and \ref{theo:sem=archimedian}, $X = E_{8}/\Gamma_{8}$ for some group $\Gamma_{8} \leq H_{8} \leq {\rm Aut}(E_{8})$. By the same arguments as above, the number of $G_{8}/\Gamma_{8}$-orbits of vertices of $X$ is three, where $G_8$ is generated by $H_{8} = \langle x\mapsto x+(w_{1,0}-w_{0,0}), x\mapsto x+(w_{0,2}-w_{0,0}), x\mapsto x+(w_{-1,2}-w_{0,0})\rangle$  and the mapping  $x\mapsto -x$. (Here the three $G_{8}$-orbits of vertices of $E_8$ are $\{u_{i,j} \, : \, i, j\in \mathbb{Z}\}$, $\{v_{i,j} \, : \, i, j\in \mathbb{Z}\}$, $\{w_{i,j} \, : \, i, j\in \mathbb{Z}\}$.
In this case, we take the origin at the middle point of the line segment joining $w_{0,0}$ and $w_{0,1}$.) Thus, the number of ${\rm Aut}(X)$-orbits of vertices is $\leq 3$.

If $X$ is a semi-equivelar map of type $[3^1, 12^2]$ on the torus then, by Theorems \ref{theo:plane} and \ref{theo:sem=archimedian}, $X = E_{9}/\Gamma_{9}$ for some group $\Gamma_{9} \leq H_{9} \leq {\rm Aut}(E_{9})$. Again, by the same arguments, the number of $G_{9}/\Gamma_{9}$-orbits of vertices of $X$ is three, where $G_9$ is generated by $H_{9} = \langle x\mapsto x+(v_{2,0}-v_{0,0}), x\mapsto x+(v_{1,1}-v_{0,0}), x\mapsto x+(v_{-2,1}- v_{0,0})\rangle$  and the mapping  $x\mapsto -x$. (Here also the three $G_{9}$-orbits of vertices of $E_9$ are $\{u_{i,j} \, : \, i, j\in \mathbb{Z}\}$, $\{v_{i,j} \, : \, i, j\in \mathbb{Z}\}$, $\{w_{i,j} \, : \, i, j\in \mathbb{Z}\}$.
In this case, we take the origin at the middle point of the line segment joining $v_{0,0}$ and $v_{1,1}$.) Thus, the number of ${\rm Aut}(X)$-orbits of vertices is $\leq 3$. This completes the proof of part (b).

Finally, assume that $X$ is a semi-equivelar map of type $[4^1, 6^1, 12^1]$ on the torus then, by Theorems \ref{theo:plane} and \ref{theo:sem=archimedian}, $X = E_{10}/\Gamma_{10}$ for some group $\Gamma_{10} \leq H_{10} \leq {\rm Aut}(E_{10})$.
Let $G_{10}$ be the group generated by $H_{10} = \langle x\mapsto x+(y_{1,0}-y_{0,0}), x\mapsto x+(y_{0,2}-y_{0,0}), x\mapsto x+(y_{-1,2}-y_{0,0})\rangle$  and the mapping  $x\mapsto -x$.
By the same arguments as above, $\Gamma_{10} \unlhd G_{10}$ and the number of $G_{10}/\Gamma_{10}$-orbits of vertices of $X$ is six. (The six $G_{10}$-orbits of vertices of $E_{10}$ are $\{u_{i,j} \, : \, i, j\in \mathbb{Z}\}$, $\{v_{i,j} \, : \, i, j\in \mathbb{Z}\}$, $\{w_{i,j} \, : \, i, j\in \mathbb{Z}\}$, $\{x_{i,j} \, : \, i, j\in \mathbb{Z}\}$, $\{y_{i,j} \, : \, i, j\in \mathbb{Z}\}$ and $\{z_{i,j} \, : \, i, j\in \mathbb{Z}\}$.
Here, we take the origin at the middle point of the line segment joining $y_{0,0}$ and $y_{0,1}$.) Thus, the number of ${\rm Aut}(X)$-orbits of vertices is $\leq 6$. This proves part (c).
\end{proof}

\begin{proof}[Proof of Theorem \ref{theo:orbits-3&6}]
Consider the semi-equivelar map $M_1$ given in Example \ref{exam:torus}. By Theorems \ref{theo:plane} and \ref{theo:sem=archimedian}, $M_1 = E_6/\Theta_6$ for some subgroup $\Theta_6$ of ${\rm Aut}(E_6)$. Let $H_6$ and $G_6$ be as in the proof of Theorem \ref{theo:no-of-orbits}.
Then we know that $\Theta_6 \leq H_{6}$, $\Theta_6 \unlhd G_{6}$ and the number of $(G_6/\Theta_6)$-orbits of vertices of $M_1$ is three. It also follows from the same proof that the three orbits are $O_u := \{u_1, \dots, u_{12}\}$, $O_v := \{v_1, \dots, v_{12}\}$, $O_w := \{w_1, \dots, w_{12}\}$. Let $\mathcal{G}_1$ be the graph whose vertices are the vertices of $M_1$ and edges are the longest diagonals of 6-gons. Then $\mathcal{G}_1$ is a 2-regular graph and hence union of cycles. Observe that $C_{4}(u_{1}, u_{2}, u_{3}$, $u_{4})$, $C_{3}(v_{1}, v_{8}, v_{11})$,
$C_{12}(w_{1}, w_{2}, w_{10}, w_{7}$, $w_{8}, w_{9}, w_{5}$, $w_{6}, w_{12}, w_{3}, w_{4}, w_{11})$ are three cycles (components) in $\mathcal{G}_1$.
Clearly, ${\rm Aut}(M_1)$ acts on $\mathcal{G}_1$. So, ${\rm Aut}(M_1) \leq {\rm Aut}(\mathcal{G}_1)$. Since $u_1, v_1, w_1$ are in components of different sizes in $\mathcal{G}_1$, it follows that they are in different ${\rm Aut}(\mathcal{G}_1)$-orbits. Hence $u_1,v_1, w_1$ are in different ${\rm Aut}(M_1)$-orbits. This implies that the ${\rm Aut}(M_1)$-orbits are $O_u, O_v$ and $O_w$. In other words, $M_1$ has exactly three ${\rm Aut}(M_1)$-orbits of vertices.

Consider the semi-equivelar map $M_3$ given in Example \ref{exam:torus}. By Theorems \ref{theo:plane} and \ref{theo:sem=archimedian}, $M_3 = E_8/\Theta_8$ for some subgroup $\Theta_8$ of ${\rm Aut}(E_6)$. Let $\mathcal{G}_3$ be the graph whose vertices are the vertices of $M_3$ and edges are diagonals of 4-gons. Then $\mathcal{G}_3$ is a 2-regular graph and hence union of cycles. Here $C_{6}(a_{1}, a_{4}, a_{2}$, $a_{5}, a_{3}, a_{6})$, $C_{4}(b_{1}, b_{4}, b_{7}, b_{10})$, $C_{12}(c_{1}, c_{5}, c_{8}, c_{11}$, $c_{3}, c_{4}, c_{7}$, $c_{10}, c_{2}, c_{6}, c_{9}, c_{12})$ are three cycles (components) in $\mathcal{G}_3$. Therefore, by the same arguments as above, $M_3$ has exactly three ${\rm Aut}(M_3)$-orbits of vertices.

Consider the semi-equivelar map $M_4$ given in Example \ref{exam:torus}. By Theorems \ref{theo:plane} and \ref{theo:sem=archimedian}, $M_4 = E_9/\Theta_9$ for some subgroup $\Theta_9$ of ${\rm Aut}(E_9)$. Let $\mathcal{G}_4$ be the graph whose vertices are the vertices of $M_4$ and edges are the longest diagonals of 12-gons. Then $\mathcal{G}_4$ is a 2-regular graph and hence union of cycles. Here $C_{12}(u_{1}, u_{12}, u_{21}$, $u_{8}, u_{9}, u_{20}$, $u_{5}, u_{16}, u_{17}$, $u_{4}, u_{13}, u_{24})$, $C_{6}(v_{1}, v_{10}, v_{17}$, $v_{2}, v_{9}, v_{18})$, $C_{4}(w_{1}, w_{6}, w_{3}, w_{8})$ are cycles in $\mathcal{G}_4$. Therefore, by the same arguments as above, $M_4$ has exactly three ${\rm Aut}(M_4)$-orbits of vertices.


Now, consider the semi-equivelar map $M_2$ given in Example \ref{exam:torus}. By Theorems \ref{theo:plane} and \ref{theo:sem=archimedian}, $M_2 = E_7/\Theta_7$ for some subgroup $\Theta_7$ of ${\rm Aut}(E_7)$. Let $H_7$ and $G_7$ be as in the proof of Theorem \ref{theo:no-of-orbits}.
Then we know that $\Theta_7 \leq H_{7}$, $\Theta_7 \unlhd G_{7}$ and the number of $(G_7/\Theta_7)$-orbits of vertices of $M_2$ is three. It also follows from the same proof that the three orbits are $O_a := \{a_1, \dots, a_{8}\}$, $O_b := \{b_1, \dots, b_{8}\}$, $O_c := \{c_1, \dots, c_{8}\}$.

We call an edge $uv$ of $M_2$ {\em nice} if at $u$ (respectively, at $v$) three 3-gons containing $u$ (respectively, $v$) lie on one side of $uv$ and one on the other side of $uv$. (For example, $c_{6}c_{7}$ is nice). Observe that there is exactly one nice edge in $M_2$ through each vertex.
Let $\mathcal{G}_2$ be the graph whose vertices are the vertices of $M_2$ and edges are the nice edges and the long diagonals of $6$-gons (this graph was first constructed in the proof of \cite[Lemma 4.3]{DM2017}). Then $\mathcal{G}_2$ is a $2$-regular graph and hence is a disjoint union of cycles. Again, we can assume that ${\rm Aut}(M_2) \leq {\rm Aut}(\mathcal{G}_2)$.
Observe that $C_{4}(b_{1}, b_{8}, b_{2}, b_{7})$, $C_{8}(a_{1}, a_{3}, a_{5}, a_{7}, a_{2}, a_{4}, a_{6}, a_{8})$,
$C_{8}(c_{1}, c_{3}, c_{5}, c_{8}, c_{2}, c_{4}, c_{6}, c_{7})$ are three cycles (components) in $\mathcal{G}_2$.
Since $a_1$ (resp. $c_1$) and $b_1$ are in components of different sizes in $\mathcal{G}_2$, it follows that they are in different ${\rm Aut}(\mathcal{G}_2)$-orbits and hence in different ${\rm Aut}(M_2)$-orbits.
If possible, suppose there exists $\alpha\in {\rm Aut}(M_2)$ such that $\alpha(a_1)=c_1$.
Then $\alpha$ maps the neighbours of $a_1$ to neighbours of $c_1$ in $M_2$. So, $\alpha(\{a_3, b_1, b_3, c_2, c_4\}) = \{a_2, a_3, b_1, b_7, c_3\}$.
Since $\alpha\in {\rm Aut}(\mathcal{G}_2)$, any $b_i$ maps to some $b_j$ by $\alpha$. So,
$\alpha(\{b_1, b_3\}) = \{b_1, b_7\}$. This is not possible since $b_1b_3$ is not an edge of $\mathcal{G}_2$ but $b_1b_7$ is an edge of $\mathcal{G}_2$. Thus, there does not exist any  $\alpha\in {\rm Aut}(M_2)$ such that $\alpha(a_1)=c_1$. Therefore, $a_1$ and $c_1$ are in different ${\rm Aut}(M_2)$-orbits. This implies that the number of ${\rm Aut}(M_2)$-orbits of vertices of $M_2$ is three. These four examples proof part (a).

Finally, consider the semi-equivelar map $M_5$ given in Example \ref{exam:torus}. By Theorems \ref{theo:plane} and \ref{theo:sem=archimedian}, $M_5 = E_{10}/\Theta_{10}$ for some group $\Theta_{10} \leq H_{10} \leq {\rm Aut}(E_{10})$.
Again, let $G_{10}$ be as in the proof of Theorem \ref{theo:no-of-orbits}.
Then $\Theta_{10} \unlhd G_{10}$ and the number of $(G_{10}/\Theta_{10})$-orbits of vertices of $M_5$ is six. It also follows from the same proof that the six orbits are $O_a := \{a_1, \dots, a_{12}\}$, $O_b := \{b_1, \dots, b_{12}\}$, $O_c := \{c_1, \dots, c_{12}\}$, $O_u := \{u_1, \dots, u_{12}\}$, $O_v := \{v_1, \dots, v_{12}\}$ and $O_w := \{w_1, \dots, w_{12}\}$.
Let $\mathcal{G}_5$ be the graph whose vertices are the vertices of $M_5$ and edges are the diagonals of 4-gons and the longest diagonals of $12$-gons. Then
\begin{align*}
  \mathcal{G}_5 = & ~ C_{12}(v_1, \dots, v_{12}) \sqcup C_{12}(w_1, \dots, w_{12}) \sqcup C_{6}(u_1, \dots, u_{6}) \sqcup C_{6}(u_7, \dots, u_{12})\\
   &  ~~ \sqcup C_{6}(c_1, \dots, c_{6}) \sqcup C_{6}(c_7, \dots, c_{12}) \sqcup C_{4}(a_1, \dots, a_{4}) \sqcup C_{4}(a_5, \dots, a_{8})\\
   &  ~~ \sqcup C_{4}(a_9, \dots, a_{12}) \sqcup C_{4}(b_1, \dots, b_{4}) \sqcup C_{4}(b_5, \dots, b_{8}) \sqcup C_{4}(b_9, \dots, b_{12}).
\end{align*}
Since ${\rm Aut}(M_5)$ acts on $\mathcal{G}_5$, we can assume that ${\rm Aut}(M_5) \leq {\rm Aut}(\mathcal{G}_5)$.
Suppose there exists $\alpha\in {\rm Aut}(M_5)$ such that $\alpha(a_1)=b_1$.
Then $\alpha$ maps the neighbours of $a_1$ to neighbours of $b_1$ in $M_5$. So, $\alpha(\{b_1, b_4, w_1\}) = \{a_1, a_4, c_2\}$.
Since $\alpha\in {\rm Aut}(\mathcal{G}_5)$, it follows that
$\alpha(\{b_1, b_4\}) = \{a_1, a_4\}$. This implies that $\alpha(w_1)=c_2$.
This is not possible since $w_1$ and $c_2$ are in components of $\mathcal{G}_5$ of different sizes and $\alpha\in {\rm Aut}(\mathcal{G}_5)$. Thus, there does not exist any $\alpha\in {\rm Aut}(M_5)$ such that $\alpha(a_1)=b_1$ and hence $a_1$ and $b_1$ are in different ${\rm Aut}(M_5)$-orbits. Similarly, one can show that $v_1$ and $w_1$ are in different ${\rm Aut}(M_5)$-orbits and $u_1$ and $c_1$ are in different ${\rm Aut}(M_5)$-orbits.
These imply that the number of ${\rm Aut}(M_5)$-orbits of vertices of $M_5$ is exactly six.
This proves part (b).
\end{proof}



\noindent {\bf Acknowledgements:}
The second author is supported by UGC, India by UGC-Dr. D. S. Kothari Post Doctoral Fellowship (F.4-2/2006 (BSR)/MA/16-17/0012).

{\small

}

\end{document}